\documentclass[leqno,a4paper,12pt]{article}

\usepackage{latexsym,a4,calc}
\usepackage[intlimits]{amsmath}
\usepackage{amsthm,amssymb,amscd}
\usepackage{pstricks,pst-node,multido,pst-plot,graphicx}
\usepackage{multicol,makeidx}

\newtheorem{satz}{Satz}

\newtheorem{lemma}[satz]{Lemma}
\newtheorem{prop}[satz]{Proposition}
\newtheorem{theo}[satz]{Theorem}
\newtheorem{cor}[satz]{Corollary}

\newcommand{\rbox}[1]{}

\theoremstyle{definition} 

\newtheorem{rem}[satz]{Remark}

\newtheorem{dfn}[satz]{Definition}

\newtheorem{example}[satz]{Example}

\newcommand{\R}{\ensuremath{\mathbb{R}}}

\newcommand{\Z}{\ensuremath{\mathbb{Z}}}
\newcommand{\N}{\ensuremath{\mathbb{N}}}

\newcommand{\ovec}[2][k]{#2_1,\ldots,#2_{#1}}

\newcommand{\nabl}[2][]{\frac{\nabla #1}{\partial #2}}

\newcommand{\dnd}[2][]{\frac{\partial#1}{\partial#2}}

\newcommand{\ddt}{\frac{d}{\, dt}}

\newcommand{\tr}{{\ensuremath{\mathrm{trace}\,}}}

\newcommand{\vol}{{\ensuremath{\mathrm{vol}}}}

\newcommand{\Hom}{{\ensuremath{\mathrm{Hom}}}}

\title{\textbf{ On the Evolution Equation for \\  Magnetic Geodesics}}

\author{Dennis Koh}

\begin{document}
\pagenumbering{roman}
\maketitle
  
\parindent0cm
\sloppy

\begin{abstract} Orbits of charged particles under the effect of a magnetic field are mathematically described by magnetic geodesics. They appear as solutions to a system of (nonlinear) ordinary differential equations of second order. But we are only  interested in periodic solutions. To this end, we study  the corresponding  system of (nonlinear) parabolic equations for closed magnetic geodesics and, as a main result, eventually prove the existence of  long time solutions. As generalization one can consider a system of elliptic nonlinear partial differential equations whose solutions describe the orbits of closed $p$-branes under the effect of a "generalized physical force". For the corresponding evolution equation, which is a system of parabolic nonlinear partial differential equations associated to the elliptic PDE, we can establish existence of short time solutions.
\end{abstract}

\section{Introduction}
\pagenumbering{arabic}

\textbf{General Assumptions.} All occurring manifolds, maps and tensors are assumed to be smooth unless otherwise stated. Also we explicitely note that all manifolds are assumed to be without boundary. 
Furthermore, we will frequently make use of "Einstein's sum convention": All sum signs are omitted if an index appears twice regardless of the position of the indices. Then one has to think of these sums to be performed. For example,
  $a_i  b_i$ is to mean $\sum_{i} a_i b_i$ and $R^l_{kij}g_{ln}g^{km}$ is to mean $\sum_{k,l}R^l_{kij}g_{ln}g^{km}$. Deviations of this convention will be made explicit by writing out the sum signs.\\
  
In this paper we investigate a certain evolution equation, which is motivated from String theory.  Namely, let $(\Sigma,g)$ and $(M,G)$ be  Riemannian manifolds, let $\Sigma$ be compact and oriented, $p=\dim(\Sigma)$. Furthermore, let $Z\in\Gamma(\mathrm{Hom}(\Lambda^{\! p} TM,TM))\cong \Gamma(\Lambda^{\! p}T^*\! M\otimes TM)$ be a tensor field such that \begin{equation}\label{Krummung}
\Omega:=G(\cdot,Z(\cdot))
\end{equation} is a closed $(p+1)$-form. Such a tensor field $Z\in\Gamma(\mathrm{Hom}(\Lambda^{\! p} TM,TM))$ coming from a $(p+1)$-form is called a \emph{$p$-force} and in the special case $p=1$ a \emph{Lorentz force}. For a map $\varphi\in C^2(\Sigma,M)$,  consider the system of nonlinear elliptic partial differential equations 
\begin{equation}\label{haha}
\tau(\varphi)=Z((d\varphi)^{\underline{p}}(\vol_g^\sharp)),
\end{equation}
which is just the Euler-Lagrange equation coming from a modified energy functional (see \cite{koh}, Chapter 2 and 3).  
In terms of a positively oriented local orthonormal frame $\{e_i\}$ of $\Sigma$, $\tau(\varphi)$ and $(d\varphi)^{\underline{p}}(\vol_g^\sharp)$ are given by  $\tau(\varphi)=(\nabla_{e_i}d\varphi)(e_i)$ and $(d\varphi)^{\underline{p}}(\vol_g^\sharp)=d\varphi(e_1)\wedge\ldots\wedge d\varphi(e_p)$, respectively. $(d\varphi)^{\underline{p}}(\vol_g^\sharp)$ can be interpreted as  vectorial volume element of $\Sigma$, being pushed forward to $M$. Now, if $p\geq 1$ is a positive integer and $\Sigma$ is connected, then a solution to  equation (\ref{haha}) describes the orbit of a closed $(p-1)$-brane  under the effect of a field strength $\Omega$. From elliptic regularity theory (see Appendix \ref{appB}, Theorem \ref{ellireg}) it follows that any $C^2$ solution of $(\ref{haha})$ is automatically $C^\infty$. The tensor field $Z:M\to\mathrm{Hom}(\Lambda^{\! p} TM,TM) $ can be interpreted as a physical force influencing the motion of the closed $(p-1)$-brane. In String theory a $p$-brane is an ''extended object'' of dimension $p$. That is, a $0$-brane corresponds to a particle, a $1$-brane to a string, $2$-brane to a membrane etc. In the special  case
$p=\dim(\Sigma)=1$,  locally we can parametrize $\Sigma$ by arc length, that is, we can always find local coordinates $\Phi: (-\epsilon,\epsilon)\to U\subset \Sigma , s\mapsto \Phi(s)$  of $\Sigma$ such that for the norm of the corresponding coordinate vector field  $g(\dnd{s},\dnd{s})=1$ holds.
 With respect to such coordinates, for $\varphi=\gamma: \Sigma\to M, s\mapsto \gamma(s)$, putting $\gamma'=\dnd[\gamma]{s}=d\gamma(\dnd{s})$,  equation (\ref{haha}) reduces to the \emph{equation for  magnetic geodesics}

\begin{equation}\label{huhu}
\frac{\nabla}{\partial s}\gamma' = Z(\gamma').
\end{equation}
  In this case a solution to the equation describes the  orbit of a charged particle under the effect of a magnetic field.  $Z$ can be interpreted as Lorentz force. For more on this topic, see e.g. \cite{bahri}, \cite{burns}, \cite{contreras}, \cite{taim} and the references therein. From now on,  whenever $\Sigma\cong S^1$, equations like (\ref{huhu}) and expression like $\gamma'=\dnd[\gamma]{s}=d\gamma(\dnd{s})$ are to be understood with respect to arc length parametrization. The problem of the existence of closed magnetic geodesics was originally posed by Novikov in early 1980s who, in particular, demonstrated its crucial difference from the closed geodesic problem and also introduced high-dimensional
analogs of it (see \cite{nov}, in the article these p-branes are also discussed).
 \\

To the elliptic PDE (\ref{haha}) one can associate an evolution equation and study the long time  behavior of its geometrical flow. Namely, we consider, for a  map $\varphi :\Sigma\times [0,T)\to M$, setting $\varphi_t(x)=\varphi(x,t)$, the initial value
problem of a system of nonlinear parabolic partial differential
equations 
\begin{equation}\label{jaja}
\left\{\begin{array}{lc }
\tau(\varphi_t)(x)=Z((d\varphi_t)^{\underline{p}}(\vol_g^\sharp))+\dnd[\varphi_t]{t}{(x)},\quad\quad
(x,t)\in\Sigma\times (0,T),\\
\varphi(x,0)=f(x),
\end{array}\right.
\end{equation} where  $\tau(\varphi_t)=\tr(\nabla d\varphi_t)$ and $f\in C^\infty(\Sigma,M)$ is a map
given as initial condition. One hopes that this problem possesses a solution for $T=\infty$ and that the limit map $\varphi_\infty=\lim_{t\to\infty}\varphi_t :\Sigma\to M$, provided that it exists, is a solution to (\ref{haha}). 
We will show that it depends on the initial condition $f$ whether the limit map $\varphi_\infty$, provided that it exists, satisfies equation (\ref{haha}) or not.
In $\dim(\Sigma)=p=1$ the above parabolic PDE (\ref{jaja}) is called the \emph{Evolution Equation for  Magnetic Geodesics}. A general introduction to nonlinear evolution equations and methods to prove existence of long time solutions are given in \cite{haraux}.
The method to find a solution to an  elliptic PDE by solving an associated parabolic (evolution) equation has been applied by Eells and Sampson  to prove the existence of harmonic maps. In the literature it is known as \emph{heat flow method}. We discuss this method in Section \ref{gmethod} and provide some Bochner-type formulas for later purposes. Good references to this topic are \cite{eecoll}, \cite{nish}, \cite{ura} and \cite{tay}.  The geometrical flow approach was also used to prove the existence of closed geodesics, i.e. to the classical problem. In particular, a new
proof of such hard result as of the Lyusternik-Schnirelmann theorem was obtained by Grayson (see \cite{cage},\cite{gray}).
\\

\vspace{-10pt}
In Section \ref{stime} we will show short time existence of the flow. The main ingredient of the proof is the Inverse Function Theorem from functional analysis. Regardless of the dimension and the curvature of $\Sigma$ and $M$, short time existence can always be guaranteed.  For the long time existence the Bochner formulas come into play. We will use them in Section \ref{ltime} to prove long time existence of the flow in $\dim(\Sigma)=1$. The maximum principle is used to obtain good a priori estimates from the Bochner formulas for the energy densities of a solution to the initial value problem (\ref{jaja}). In this way the growth rate of the solutions, as time $t$ increases, is controlled and blow ups are prevented.\\
\vspace{-10pt}

\section{Statement of the  results}  \label{results}

\begin{theo}[Long time existence]\label{longex} Let $\Sigma=S^1$ and $(M,G)$ be a compact Riemannian manifold. Moreover let $Z\in\Gamma(\mathrm{Hom}(TM,TM))$ be a Lorentz force. Set $\gamma_t(s)=\gamma(s,t)$ and $\gamma_t'=\dnd[\gamma_t]{s}=d\gamma_t(\dnd{s})$. Then for any $C^{2+\alpha}$ map $f\in C^{2+\alpha}(S^1,M)$, there exists a unique $\gamma\in C^{2+\alpha,1+\alpha/2}(S^1\times [0,\infty),M)\cap C^{\infty}(S^1\times (0,\infty),M)$ such that
\begin{equation}\label{pama}
\left\{\begin{array}{lc }
\frac{\nabla }{\partial s}\gamma'_t(s)=Z(\gamma'_t)(s)+\dnd[\gamma_t]{t}{(s)},\quad\quad
(s,t)\in S^1\times (0,\infty),\\
\gamma(s,0)=f(s),
\end{array}\right.
\end{equation} 
 holds.
\end{theo}

\begin{theo}[Stability and uniqueness of solutions]\label{unique}  Assume that $\Sigma=S^1$. Let $(M,G)$ be a Riemannian manifold  and  $Z,Z'\in\Gamma(\Hom(TM,TM))$ be  Lorentz forces.  Let  $u,v\in C^0(S^1\times [0,T),M)\cap C^{2,1}(S^1\times (0,T),M)$. Setting $u_t(s)=u(s,t)$ and $v_t(s)=v(s,t),$ assume that $u$ satisfies the  evolution  equation for magnetic geodesics
\begin{equation}\label{paramag}
\frac{\nabla }{\partial s}\dnd[u_t]{s}(s)=Z(\dnd[u_t]{s})(s)+\dnd[u_t]{t}{(s)},\quad\quad
(s,t)\in S^1\times (0,T),
\end{equation} and similarly that $v$ satisfies $\mathrm{(\ref{paramag})}$ with $Z'$ instead of $Z$. Furthermore, assume that $Z$ and $Z'$ are bounded, i.e. $|Z|_{L^\infty(M,E)}, |Z'|_{L^\infty(M,E)}<\infty$.  Then for any $0<T_0<T$ there exists a constant $C=C(T_0)\geq 0$ such that \begin{equation}\label{stabil}|u_t-v_t|_{L^2(\Sigma,M)}^2\leq 2\pi e^{Ct}\Big(|u_0-v_0|_{L^\infty(\Sigma,M)}^2+t |Z-Z'|^2_{L^\infty(M,E)}\Big)\end{equation} holds for all $t\in [0,T_0]$. Here, $E=\mathrm{Hom}(\Lambda^{\! k}TM,TM)$,  $|Z|_{L^\infty(M,E)}=\sup_{M}\left\langle Z,Z\right\rangle^{1/2}$ and  $C=C(T_0)\geq 0$ is a nonnegative constant  depending on $T_0$ and other parameters. The dependence is clarified in the course of the proof. In particular, $u_0= v_0$  and $Z=Z'$ imply $u= v$ throughout $\Sigma\times [0,T)$.
\end{theo}

\begin{cor} Let $\Sigma,M,Z,Z',u,v$ and the assumptions on them as above in Theorem \ref{unique}. If in addition $M$ is compact, then \emph{(\ref{stabil})} holds for all $t\in [0,T)$.
 \end{cor}
\begin{proof}  Since $M$ is compact, the ball $B(0,r)$ in the  proof of Theorem 2 can be chosen  such that $M\subset B(0,r)\subset\R^q$. The boundedness of $Z$ and $Z'$ (need not to be assumed, but follows from the compactness of $M$)  implies that the energy densities $e(u_t)$ and $e(v_t)$ can be globally estimated on $[0,T)$ by Proposition \ref{engyesti}. Consequently the constant $C\geq 0$ from the above proof can be chosen to be independent of $T_0$.\end{proof}

\begin{cor}\label{liegroup} Let  $(M,G)$ be a  Riemannian manifold  and $Z\in\Gamma(\Hom(TM,TM))$ be a Lorentz force. Furthermore, let $H$ be a discrete group of isometries of $(M,G)$ acting properly discontinuously  on $M$. If $Z$ is $H$-invariant, i.e. $dh\circ Z=Z\circ dh$ for all $h\in H$, and the quotient $M/H$ is compact, then for any $C^{2+\alpha}$ map $f\in C^{2+\alpha}(S^1,M)$, there exists a unique long time solution $\gamma\in C^{2+\alpha,1+\alpha/2}(S^1\times [0,\infty),M)\cap C^{\infty}(S^1\times (0,\infty),M)$ to the \emph{IVP} $\mathrm{(\ref{pama})}$ in $M$.
\end{cor}

\begin{proof} The result follows immediately  by pushing the entire initial value problem in $M$ down to $M/H$ (equipped with the unique structure of a Riemannian manifold). Applying Theorem \ref{longex} to the corresponding initial value problem in $M/H$ yields a unique solution which can be lifted to a unique solution to the original initial value problem on $M$.
\end{proof}

\begin{example} Let $(M,G)$ be the three-dimensional Euclidean space $\R^3$ and
$B\in\R^3$ be a parallel vector field in $\R^3$, (all tangent spaces of $\R^3$ are identified by parallel transport). We define a skew-symmetric bundle homomorphism $Z:T\R^3\to T\R^3$, $Z(v)=v\times B$ for all $v\in\R^3$, by means of the vector product. From $\nabla Z=0$ we see that, in fact, $Z$ comes from a closed two-form $\Omega$ via (\ref{Krummung}). Since $Z$ is translation-invariant and  the three-torus $T^3=\R^3/\Z^3$ is compact, we deduce long time existence of solutions to the IVP (\ref{pama}) from Corollary \ref{liegroup}. This holds more generally for any $\Z^3$-invariant Lorentz force $Z$.
\end{example}

\begin{rem} The compactness of $\Sigma$ in Theorem 1 cannot be dropped. In general, the lifetime $T$ of a solution to the IVP (\ref{pama}) for non-compact $\Sigma$ may be finite. For example, let $\Sigma=M=\R$ and $T>0$ be a positive number. Consider the function $u: \R\times [0,T)\to \R$ defined by
\[u(s,t)=\frac{s}{T-t}.\] This is a smooth function on $\R\times [0,T)$ which blows up as $t\to T$. Let $Z:T\R\to T\R$ be the bundle homomorphism defined by $Z_s(v):=-sv$, $(s,v)\in \R\times\R$. The function $u$ solves the IVP (\ref{pama}) on $\R\times (0,T)$, with initial condition $u(s,0)=s/T$ and the above defined $Z$. In this case the parabolic equation just reads \[v'=-uv+\dot{u},\quad \mbox{on }\;  \R\times (0,T),\] where $\dot{u}=\dnd[u]{t}$, $v=\dnd[u]{s}$ and $v'=\frac{\partial^2 u}{\partial s^2}$. 
This demonstrates that the lifetime of solutions to the IVP (\ref{pama}) can be finite for non-compact $\Sigma$.  
\end{rem}

\begin{cor}
 Let $\Sigma=S^1$ and $(M,G)$ be a Riemannian manifold. Furthermore, let $Z\in\Gamma(\mathrm{Hom}(TM,TM))$ be a Lorentz force and  $\gamma\in C^{2,1}(S^1\times [0,T),M)\cap C^\infty(S^1\times (0,T),M)$ be a solution to the \emph{IVP} $\mathrm{(\ref{pama})}$, where $T=\sup\,\{ t\in [0,\infty)\,|\, (\ref{pama}) \mbox{ has a solution in } S^1\times [0,t]\} $. Set $\gamma_t(s)=\gamma(s,t)$. If $T<\infty$, then for any compact subset $K\subset M$ and any $0<T_0<T$, there exists a $t\in (T_0,T)$ such that $\gamma_t(S^1)\cap (M-K)\neq\emptyset$. Said in words: If the lifetime $T$ of  a solution $\gamma$ is finite, then it leaves any compact subset of $M$, or equivalently, if a solution $\gamma$ stays its entire life in a compact set, then its lifetime $T=\infty$.
 \end{cor}
 
 \begin{proof} Let $T<\infty$ and assume that the conclusion is false. Then there  exist a compact subset $K\subset M$ such that $\gamma(S^1\times [0,T))\subset K$ holds. Set $E=\mathrm{Hom}(\Lambda^{\! k}TM,TM)$.
Now, we proceed quite literally as in the proof of Proposition \ref{fundesti} and obtain
\[  |\gamma(\cdot,t)|_{C^{2+\alpha}(S^1,M)}+\Big|\dnd[\gamma]{t}(\cdot,t)\Big|_{C^{\alpha}(S^1,M)}\leq C.\]
Here, $C=C(\Sigma,K,M,Z,f,\alpha,T)$ is a constant only depending on $\Sigma,K, M,Z, f,\alpha$ and $T$.  The only difference is that in all estimates (energy estimates etc.) one has to replace all occurrences of $|\cdot|_{L^\infty(M,E)}$ by $|\cdot|_{L^\infty(K,E|_K)}$. Obviously   (\ref{bounded}) holds since $\gamma(S^1\times [0,T))\subset K$. Then similarly as in the proof of Theorem \ref{longex} one extends the solution to $S^1\times [0,T+\epsilon]$ (for $\epsilon>0$ sufficient small) and produces a contradiction to the definition of $T$. \end{proof}

\section{The heat flow method}  \label{gmethod}

\textbf{Notational convention.} Throughout the whole paper let $(\Sigma^k,g)$ and  $(M^n,G)$ be Riemannian manifolds. Furthermore, let $\Sigma$ be compact and oriented and let  $Z\in\Gamma(\mathrm{Hom}(\Lambda^{\!  k} TM,TM))$ be a $k$-force determined by some closed $(k+1)$-form $\Omega$ as in (\ref{Krummung}). Henceforth, we abbreviate $Z((d\varphi)^{\underline{k}})=Z((d\varphi)^{\underline{k}}(\vol_g^\sharp))$ and $Z_\varphi((d\varphi)^{\underline{k}})=Z_{\varphi}((d\varphi)^{\underline{k}}(\vol_g^\sharp))$. For the sake of simplicity all appearing metrics and covariant derivatives are denoted by $\langle\cdot,\cdot\rangle$ and $\nabla$, respectively.\\

In 1964 Eells and Sampson proved the existence of harmonic maps (see \cite{eecoll}) by
the heat flow method, that is, they demonstrated that the time limit of the solution to an associated evolution equation is a harmonic map.  We would like to use this technique to prove the
existence of a solution to equation (\ref{haha}) above. It turns out that in general this method does not yield a solution to our problem. On the contrary, we will see that the solvability rather depends on the initial value for the associated evolution equation. However, short time existence of  solutions to the associated evolution equation can always be shown, regardless of the dimension of $(\Sigma,g)$ and $(M,G)$ and without making any further assumptions, excepting that $\Sigma$ is required to be compact and oriented. On the other hand, only if $\dim(\Sigma)=1$ and assuming that $M$ is compact, we are able to verify existence of long time solutions. So, we consider
for a map $\varphi :\Sigma\times [0,T)\to M$, setting $\varphi_t(x)=\varphi(x,t)$, the initial value
problem (IVP) for the  system of nonlinear parabolic partial differential
equations \rbox{iniprob}
\begin{equation}\label{iniprob}
\left\{\begin{array}{lc }
\tau(\varphi_t)(x)=Z((d\varphi_t)^{\underline{k}})(x)+\dnd[\varphi_t]{t}{(x)},\quad\quad
(x,t)\in\Sigma\times (0,T),\\
\varphi(x,0)=f(x),
\end{array}\right.
\end{equation} where  $\tau(\varphi_t)=\tr(\nabla d\varphi_t)$ and $f\in C^\infty(\Sigma,M)$ is a map
given as initial condition. We assume that
\[\varphi\in C^0(\Sigma\times [0,T),M)\cap
C^\infty(\Sigma\times(0,T),M).\]

Before going into the details of the proofs, we compute the following.
\begin{example} Let $\Sigma=S^1$ the unit circle and $M=T^2=S^1\times S^1$ the two-dimensional  standard torus with the natural induced metrics. 
Then for a map $\gamma :S^1\times [0,\infty)\to M$, setting $\gamma_t(s)=\gamma(s,t)$, the IVP (\ref{iniprob}) takes the form
\begin{equation}
\left\{\begin{array}{lc }
\nabl{s}\gamma_t'(s)=Z(\gamma_t')(s)+\dnd[\gamma_t]{t}{(s)},\quad\quad
(s,t)\in S^1\times (0,\infty),\\
\gamma(s,0)=c(s),  \tag{*}
\end{array}\right.
\end{equation}  where $\gamma_t'(s)=\dnd[\gamma]{s}(s,t)$ and $c:S^1\to T^2$ is a smooth initial curve. Let $\hat{M}=S^1\times\R\subset\R^3$ be the standard cylinder with  metric induced from $\R^3$ and, denoting the standard coordinates of $\R^3$ by $(x,y,z)$, let the $z$-axis be the axis of symmetry. For the radial vector field $\hat{B}:\R^3\to\R^3$ , given by
\[\hat{B}:(x,y,z)^t \mapsto (x,y,0)^t,\] we define a skew-symmetric bundle homomorphism $\hat{Z}:T\hat{M}\to T\hat{M}$ by $\hat{Z}(v)=v\times \hat{B}$ by means of the vector product of $\R^3$, (all tangent spaces of $\R^3$ are identified by parallel transport). We note that $\nabla \hat{Z}=0$, implying that $\hat{Z}$ defines a closed $2$-form $\hat{\Omega}$ via (\ref{Krummung}), and consider for a map $\gamma :S^1\times [0,\infty)\to \hat{M}\subset\R^3$ the initial value problem
\begin{equation}
\left\{\begin{array}{lc }
\nabl{s}\gamma_t'(s)=\hat{Z}(\gamma_t')(s)+\dnd[\gamma_t]{t}{(s)},\quad\quad
(s,t)\in S^1\times (0,\infty),\\
\gamma(s,0)=c(s). \tag{**}
\end{array}\right.
\end{equation}  Since $\hat{B}$ is invariant under $z$-translations,  $\hat{Z}$ descends to a well-defined parallel skew-symmetric bundle homomorphism $Z: TM\to TM$ on the Torus $M=\hat{M}/\!\!\sim\, = S^1\times S^1$, regarded as quotient of $\hat{M}$ by moding out the $\Z$-action on the second factor of $\hat{M}=S^1\times \R$. Hence, the entire initial value problem $(**)$ on the cylinder $\hat{M}$ descends to a corresponding initial value problem $(*)$ on the torus $M=T^2$. So, for simplicity we will do all our computations on the cylinder $\hat{M}$. Passing to the quotient $M=\hat{M}/\!\!\!\sim$ then yields a corresponding result for the torus. Expressing $\gamma_t(s)$ and $\hat{B}$ in cylindrical coordinates
\[\gamma_t(s)=\left(
\begin{array}{c}
\cos(\varphi(s,t))\\\sin(\varphi(s,t))\\z(s,t)\\
\end{array}
\right)\quad\mbox{and}\quad 
\hat{B}(r,\varphi,z)=\left(
\begin{array}{c}
r\cos(\varphi)\\ r\sin(\varphi)\\0\\
\end{array}
\right), \] $r\in(0,\infty)$, $\varphi\in(-\pi,\pi)$, $z\in(-\infty,\infty)$, a straightforward computation shows that, for functions $\varphi,z : S^1\times [0,\infty)\to\R$,   $(**)$ is equivalent to the following system of partial differential equations
\begin{equation}
\left\{\begin{array}{lc }
\varphi''(s,t)=z'(s,t)+\dot{\varphi}(s,t),\quad\quad
(s,t)\in [0,2\pi]\times (0,\infty),\\z''(s,t)=-\varphi'(s,t)+\dot{z}(s,t),\quad\quad
(s,t)\in [0,2\pi]\times (0,\infty),\\
\varphi(s,0)=\varphi_0(s),\\z(s,0)=z_0(s). \tag{+}
\end{array}\right.
\end{equation} Here, we identify $S^1\cong \R/2\pi\Z$, i.e. we regard $\varphi$ and $z$ as functions defined on $\R\times [0,\infty)$, which are $2\pi$-periodic in the first argument. Furthermore, we  abbreviate $\varphi''=\frac{\partial^2\varphi}{\partial s^2}$, $\varphi'=\dnd[\varphi]{s}$ and $\dot{\varphi}=\dnd[\varphi]{t}$ (in the same way for $z$) and $\varphi_0,\,z_0$ are initial conditions. Now, let us explicitely calculate the flow for the initial conditions 
\begin{eqnarray*}
\textrm{a)} \;\left\{\begin{array}{l}\varphi_0(s)=A\cos(s)\\z_0(s)=B\sin(s)\end{array}\right.\quad 
\mbox{ and }\quad \textrm{b)}\;\left\{\begin{array}{l}\varphi_0(s)=s\\z_0(s)=\mu\cos(s),\end{array}\right. 
 \end{eqnarray*} where $\mu,A,B\geq 0$ are nonnegative numbers and  the function $\varphi_0$ from initial condition b)  is to be understood as being defined on $[0,2\pi]$; in terms of $\gamma_0(s)=(\cos(s),\sin(s), \mu\cos(s))$ we see that b) is a well-defined smooth initial condition $\gamma_0: S^1\cong\R/2\pi\Z\to S^1\times \R$. To this end, let us introduce the complex variable $\xi=\varphi+iz$. Here, $i$ denotes the imaginary unit.  Then  system (+) reduces to a single partial differential equation
 
 \begin{equation}
\left\{\begin{array}{lc }
\dot{\xi}(s,t)=\xi''(s,t)+i\xi'(s,t),\quad\quad
(s,t)\in [0,2\pi]\times (0,\infty),\\
\xi(s,0)=\varphi_0(s)+iz_0(s). \tag{++}
\end{array}\right.
\end{equation} To solve this we try a power series ansatz
\[\xi(s,t)=\sum_{n=0}^\infty a_n(s)t^n.\] Plugging this into (++) yields the following recursion formula for the coefficients $a_n$ for all $n\geq 1$:
\begin{equation}\label{rekursion}
 a_{n}=\frac{ a''_{n-1}+i a'_{n-1}}{n},\quad\quad a_0=\varphi_0+iz_0.  \tag{R}
\end{equation}

ad a): If $a_0(s,t)=A\cos(s)+iB\sin(s)$,  for $n\geq 1$ we get
\[a_n(s)=\frac{(A+B)}{2}\frac{(-2)^n}{n!}\exp(is),\] and consequently,
\begin{align*}
\xi(s,t)&=a_0(s)+\sum_{n=1}^\infty \frac{(A+B)}{2}\frac{(-2)^n}{n!}\exp(is)\\
&=A\cos(s)+iB\sin(s)-\frac{(A+B)}{2}\exp(is)+\frac{(A+B)}{2}\exp(is)\exp(-2t)\\
&=\frac{(A-B)}{2}\exp(-is)+\frac{(A+B)}{2}\exp(is)\exp(-2t).
\end{align*} We see that the limit as $t\to\infty$ exists, namely 
\[\xi_\infty(s)=\lim_{t\to\infty}\xi(s,t)=\frac{(A-B)}{2}\exp(-is).\] Also one readily verifies that $\xi''_\infty+i \xi'_\infty=0$ holds, i.e. on the torus $T^2=\hat{M}/\!\!\sim$ the corresponding loop $\gamma_\infty=\lim_{t\to\infty}\gamma_t:S^1\to M=T^2$ satisfies the equation for magnetic geodesics
\[\nabl{s}\gamma'_\infty=Z(\gamma'_\infty).\]
 
 ad b): If $a_0(s,t)=s+i\mu\cos(s)$,   we get $a_1(s)=i(1-\mu\exp(is))$, and for $n\geq 2$ 
 \[a_n(s)=\frac{i\mu}{2}\frac{(-2)^n}{n!}\exp(is),\]
  and thus,
\[\xi(s,t)=s+it+i\mu\cos(s)+\frac{i\mu}{2}\exp(is)\Big[\exp(-2t)-1\Big].\] On the torus $T^2=\hat{M}/\!\!\sim$ the subsequence $\{\xi(s,2\pi n)\}_{n\geq 0}$  corresponds to a constant sequence, namely to a loop $\gamma_\infty: S^1\to T^2$, surrounding the neck of the torus. (see \textsc{Figure 4.1}) The limit of any other convergent subsequence is just a translation of that loop $\gamma_\infty$ along the "soul" of the torus, i.e. a translation in $t$-direction. However, since $\xi''+i\xi'=i\neq 0$, we see that a limit loop $\gamma_\infty$ can never satisfy the equation for magnetic geodesics in contrast to case a). \\

\begin{pspicture}(-3,-3)(3,3)
\rput(0.3,2){ a): $A\neq B$}
\psellipse(0.5,-0.4)(3,2)
\psbezier[linewidth=1pt]{-}(-1.3,-0.2)(-0.9,-0.9)(1.8,-0.9)(2.1,-0.2)
\psbezier[linewidth=1pt]{-}(-1.07,-0.4)(-0.7,0.3)(1.6,0.3)(1.87,-0.4)
\psellipse[linecolor=red](0.5,-1.55)(1.2,0.7)
\psellipse[linestyle=dashed,linecolor=red](0.5,-1.55)(1,0.6)
\psellipse[linestyle=dashed,linecolor=red](0.5,-1.55)(0.7,0.5)
\pscircle[linecolor=red](0.5,-1.55){0.4}
\rput(-1.1,-1.2){\textcolor{red}{\footnotesize{$t=0$}}}
\rput(0.5,-1.55){\textcolor{red}{\tiny{$t=\infty$}}}

\rput(9.3,2){ b): $\mu=0$}
\psellipse(9.5,-0.4)(3,2)
\psbezier[linewidth=1pt]{-}(7.7,-0.2)(8.1,-0.9)(10.8,-0.9)(11.1,-0.2)
\psbezier[linewidth=1pt]{-}(7.93,-0.4)(8.3,0.3)(10.6,0.3)(10.87,-0.4)
\psbezier[linecolor=red,linewidth=1pt]{-}(9.6,-0.72)(10.1,-1.2)(10.1,-1.9)(9.6,-2.38)
\psbezier[linecolor=red,linestyle=dashed,linewidth=1pt]{-}(9.6,-0.72)(9.1,-1.2)(9.1,-1.9)(9.6,-2.38)
\rput(10.7,-1.2){\textcolor{red}{\footnotesize{$t=2\pi n,$}}}
\rput(10.63,-1.5){\textcolor{red}{\footnotesize{$n\in \N$}}}
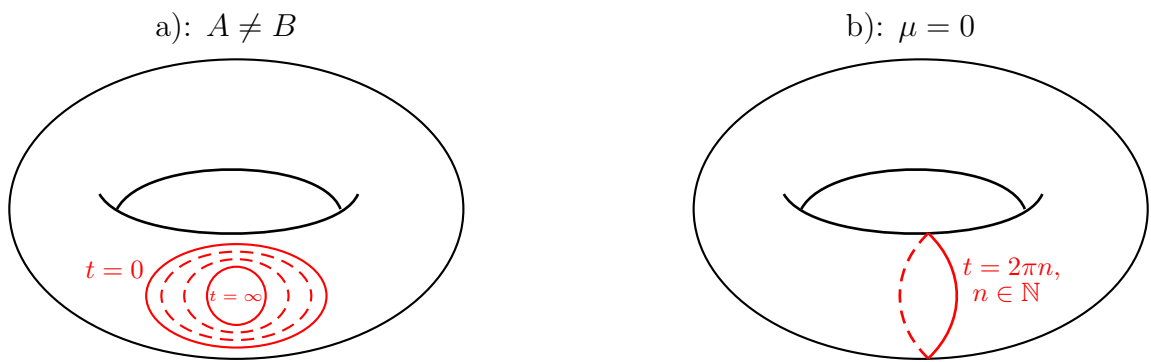
\rput(5,-3.5){\textsc{figure 4.1.}  The flow of the evolution equation}
\end{pspicture} \\\\

We may summarize as follows:\\
 On the torus we have computed the flow of the parabolic equation for magnetic geodesics for two families of initial conditions. For an  ellipse $c:S^1\to T^2$ as initial condition (case a)) not enclosing the neck of the torus, the limit loop $\gamma_\infty$, as $t\to\infty$, exists and is a magnetic geodesic. In the case b) when the initial curve $c:S^1\to T^2$ forms an ellipse enclosing the neck of the torus,  there exist convergent subsequences; but then a limit loop can not be a magnetic geodesic. Hence, we see that the existence of a convergent subsequence such that its limit curve satisfies the equation for magnetic geodesics depends on the initial condition. However, for the cylinder $S^1\times\R$ and the torus $S^1\times S^1$, respectively, long time existence of the flow is guaranteed for any initial condition by Theorem \ref{liegroup} and Theorem \ref{longex}, respectively.
\end{example}

In general, to show existence of solutions to the equation (\ref{haha}) one has to verify the steps of the following program:
\begin{enumerate}
\item Show existence of short time solutions to the parabolic initial value problem (\ref{iniprob}).
\item Rule out occurrence of blow ups in finite time, i.e. show  existence of long time solutions to the initial value problem (\ref{iniprob}).
\item Show convergence $\varphi_t\to \varphi_\infty$ as $t\rightarrow \infty$ .
\item If the limit $\varphi_\infty$ exists, show that $\varphi_\infty$ satisfies (\ref{haha}).
\end{enumerate}
 As seen from the above  example,  it depends on the initial condition 
whether a limit map $\varphi_\infty$, provided that it exists, is a solution to (\ref{haha}) or not. Consequently one cannot expect a general existence result for generalized harmonic maps in the sense of Eells and Sampson. So, we restrict ourselves to tackle the long time existence problem, i.e.  in the following sections we are going to carry out the first and the second issue of the previous program. The strategy is to derive some Bochner-type formulas and to use the maximum principle for parabolic equations to get a priori estimates which allow to control the growth rate of solutions to the IVP  (\ref{iniprob}). \\
The estimates for the energy densities will show that in $\dim(\Sigma)=k=1$ everything is fine. For $\dim(\Sigma)>1$ we would have to deal with "bad" terms that possibly could destroy the long time behavior of our solutions whereas short time existence can be guaranteed without any restrictions on the dimension  of $\Sigma$ and $M$.
  
For a given solution $\varphi$ of  (\ref{iniprob}) we set $\varphi_t(x)=\varphi(x,t)$ and define
\[e(\varphi_t):=\frac{1}{2}|d\varphi_t|^2,\quad\mbox{(energy density)} \] 
\[E(\varphi_t):=\int_\Sigma e(\varphi_t)\,d\vol_g, \quad\mbox{(energy)}\]
\[\kappa(\varphi_t):=\frac{1}{2}\Big|\frac{\partial\varphi_t}{\partial t}\Big|^2,
\quad\mbox{(kinetic energy density)}\]
\[K(\varphi_t):=\int_\Sigma \kappa(\varphi_t)\,d\vol_g.\quad\mbox{(kinetic energy)}\]

Now, we state a Weitzenb\"ock formula for vector bundle valued $1$-forms (see Appendix \ref{appA} a)). 
\begin{prop}[Weitzenb\"ock formula] Let $\omega$ be a $1$-form on a Riemannian manifold $(M,g)$ with values 
in a Riemannian vector bundle $(E,\nabla^E,h)$. Then 
\begin{align*}
\Delta\omega=\bar{\Delta}\omega+S_\omega.
\end{align*} Here, $S_\omega\in\Gamma(T^*\! M\otimes E)$ is given by \rbox{weitzen}
\begin{equation}\label{weitzen}
S_\omega(X)=(R(X,e_i)\omega)(e_i),
\end{equation} where  $\{e_i\}$ is a local orthonormal frame on $M$, $X\in\Gamma(TM)$ and $R$ is the curvature tensor corresponding to the connection on $T^*\! M\otimes E$ which is induced by the connections of $T^*\! M$ and $E$, respectively.
\end{prop} A proof can be found in (\cite{xin}, p. 21).

\begin{prop}[Bochner-type formulas]
Let  $\varphi\in C^0(\Sigma\times [0,T),M)\cap C^\infty(\Sigma\times (0,T),M)$ be a solution to the parabolic \emph{IVP} 
$\mathrm{ (\ref{iniprob})}$, and let $\varphi_t(x)=\varphi(x,t)$. In $\Sigma\times (0,T)$ we have,\\ 
(1)\quad $($Bochner formula for $e(\varphi_t)$$)$
\begin{align}
\dnd[e(\varphi_t)]{t}& =\Delta e(\varphi_t)-|\nabla d\varphi_t|^2+\langle R^M(d\varphi_t(e_i),
d\varphi_t(e_k))d\varphi_t(e_k),d\varphi_t(e_i)\rangle \\
& \mspace{23mu}-\langle d\varphi_t(Ric^\Sigma(e_i)),d\varphi_t(e_i)\rangle -\langle \nabla Z((d\varphi_t)^{\underline{k}}), d\varphi_t\rangle .\notag
\end{align}
(2)\quad $($Bochner formula for $\kappa(\varphi_t)$$)$
\begin{align}
\dnd[\kappa(\varphi_t)]{t}& =\Delta \kappa(\varphi_t)-|\nabla\dnd[\varphi_t]{t}|^2+\langle R^M(\dnd[\varphi_t]{t},
d\varphi_t(e_i))d\varphi_t(e_i),\dnd[\varphi_t]{t}\rangle \\
& \mspace{23mu} -\langle \nabl{t}\, Z((d\varphi_t)^{\underline{k}}), \dnd[\varphi_t]{t}\rangle .\notag
\end{align} Here, $\Delta=-\delta d$ is the Hodge-Laplacian on $C^2(\Sigma)$, 
$\nabla d\varphi_t(X,Y)=(\nabla_X d\varphi_t)(Y)$, for $X,Y\in T_x\Sigma$, is the second fundamental form of $\varphi_t$, and $Ric^\Sigma$ and $R^M$ denote, respectively, the Ricci tensor of $\Sigma$ and the curvature tensor of $M$. The family $\{e_i\}$ represents a positively oriented orthonormal basis for the tangent space at each $x\in\Sigma$. The covariant derivatives and the metrics are the natural induced ones.
\end{prop}

\begin{proof} Choose a positively oriented orthonormal frame $\{e_i\}$ near  $x\in\Sigma$ with $\nabla_{e_i}e_j\big |_x=0$. Then computing $\Delta e(\varphi_t)=\partial_{e_i}\partial_{e_i} e(\varphi_t)$ and $\Delta \kappa(\varphi_t)$ at the point $x$ and using the Weitzenb\"ock formula yields the desired equalities.
\end{proof}

\begin{rem} Since $\Sigma$ is compact, the unit sphere bundle $S\Sigma$ is also compact. Being a smooth function on $S\Sigma$, $Ric^\Sigma$ achieves its minimum on it.  Consequently there exists a constant $C$ such that $Ric^\Sigma\geq -Cg$. Namely, we can take $C:=-\underset{v\in S\Sigma}{\min}\, Ric^\Sigma(v,v)$.
\end{rem}

Now, set $E=\mathrm{Hom}(\Lambda^{\! k}TM,TM)$.

\begin{cor} \label{enestimates} Let $\varphi:\Sigma\times [0,T)\to M$ be a solution to the \emph{IVP (\ref{iniprob})} and set $\varphi_t(s)=\varphi(s,t)$. Let $Z=Z^\Omega$ be some $k$-force determined by some closed $(k+1)$-form $\Omega\in\Gamma(\Lambda^{\! k+1}T^*\! M)$ as in $\mathrm{(\ref{Krummung})}$, with $|Z|_{L^\infty(M,E)}, |\nabla Z|_{L^\infty(M,E)}<\infty$. The following holds in $\Sigma\times (0,T)$:\\

\emph{(1)}  Let $C$ be a real number such that $Ric^\Sigma\geq -Cg$. If $M$ is of nonpositive curvature $K^M\leq 0$, then
\begin{equation}
\dnd[e(\varphi_t)]{t}\leq\Delta e(\varphi_t)+2C e(\varphi_t)+2^{k-2}k|Z|^2_{L^\infty(M,E)}\, e(\varphi_t)^k.
\end{equation}

\emph{(2)} For the kinetic energy density, we have
\begin{align}
\dnd[\kappa(\varphi_t)]{t} &\leq \Delta \kappa(\varphi_t)+4|R^M\!|e(\varphi_t)\kappa(\varphi_t)+2^{k-2}k^2|Z|^2_{L^\infty(M,E)}\, e(\varphi_t)^{k-1}\kappa(\varphi_t)\\
&\mspace{23mu} +2^{1+k/2}|\nabla Z|_{L^\infty(M,E)}\, e(\varphi_t)^{k/2}\kappa(\varphi_t).\notag
\end{align} The norms are given by $|Z|_{L^\infty(M,E)}=\sup_M\langle Z,Z\rangle^{1/2} $ and $|\nabla Z|_{L^\infty(M,E)}=\sup_M\langle\nabla Z,\nabla Z\rangle^{1/2}$. Regarding the curvature tensor as $(4,0)$-tensor, the norm of $R^M$ is given by $|R^M|=\langle R^M,R^M\rangle^{1/2}$. All covariant derivatives, metrics and norms used here are the natural ones induced by the metrics $g$ and $G$.
\end{cor}

\begin{proof} Firstly recall the definition  of $(d\varphi)^{\underline{k}}$ and the $\tilde{\wedge}$-product in Appendix \ref{appA}(a). For simplicity we will denote all appearing metrics by $\langle\cdot,\cdot\rangle$. \\

ad (1): Firstly we note that, for an orthonormal frame with  $\nabla_{e_i}e_j\big |_x=0$, at $x$
\begin{align*}\langle\nabla \,Z((d\varphi_t)^{\underline{k}}),d\varphi_t \rangle &= 
\partial_{ e_i}\,\underbrace{\langle Z((d\varphi_t)^{\underline{k}}),d\varphi_t(e_i)\rangle}_{=0}-\langle Z((d\varphi_t)^{\underline{k}}),\nabla_{\! \! e_i}  d\varphi_t(e_i)\rangle\\
&=-\langle Z((d\varphi_t)^{\underline{k}}), \tr \nabla  d\varphi_t\rangle
\end{align*} holds  due to the skew-symmetry of $\Omega$. From this we get 
\begin{align*}
|\langle\nabla \,Z((d\varphi_t)^{\underline{k}}),d\varphi_t \rangle| &\leq k^{1/2} |Z||d\varphi_t|^k|\nabla d\varphi_t|\\
&\leq |\nabla d\varphi_t|^2+\frac{k}{4}|Z|^2_{L^\infty(M,E)}|d\varphi_t|^{2k}.
\end{align*} Using this estimate, the  curvature assumptions $K^M\leq 0$ and $Ric^\Sigma\geq -C g$, and the Bochner formula for the energy density $e(\varphi_t)$,   inequality (1) readily follows.\\

ad (2): From
\begin{align*}
\nabl{t} \,Z((d\varphi_t)^{\underline{k}})&= \Big(\nabla_{\! \dnd[\varphi_t]{t}} Z\Big)((d\varphi_t)^{\underline{k}})+
 Z\Big((\nabla  \dnd[\varphi_t]{t})\tilde{\wedge}(d\varphi_t)^{\underline{k-1}}\Big),
\end{align*} we see 
\begin{align*}
\Big|\Big\langle\nabl{t} \,Z((d\varphi_t)^{\underline{k}}), \dnd[\varphi_t]{t}\Big\rangle\Big| &\leq |\nabla Z||d\varphi_t|^k\Big|\dnd[\varphi_t]{t}\Big|^2+k|Z||d\varphi_t|^{k-1}\Big|\dnd[\varphi_t]{t}\Big|\Big|\nabla \dnd[\varphi_t]{t}\Big|\\
&\leq  \Big|\nabla \dnd[\varphi_t]{t}\Big|^2+ \frac{k^2}{4} |Z|^2_{L^\infty(M,E)}\,|d\varphi_t|^{2k-2}\Big|\dnd[\varphi_t]{t}\Big|^2\\
&\mspace{23mu} +|\nabla Z|_{L^\infty(M,E)}|d\varphi_t|^k\Big|\dnd[\varphi_t]{t}\Big|^2.
\end{align*}  From this estimate and  the Bochner formula for the kinetic energy density  $\kappa(\varphi_t)$ we obtain the desired inequality (2).
\end{proof}

As a special case of Corollary \ref{enestimates}, for $k=1$ we have the following.

\begin{cor}\label{energyestimates} Assume that $\Sigma=S^1$ and $Z$ is a Lorentz force. Let $\varphi=\gamma :S^1\times [0,T)\to M$ be a solution to the \emph{IVP} $\mathrm{(\ref{iniprob})}$, and set $\gamma_t(s)=\gamma(s,t)$. The following hold in $S^1\times (0,T)$:\\

\emph{(1')}  If $|Z|_{L^\infty(M,E)}<\infty$, then
\begin{equation}
\dnd[e(\gamma_t)]{t}\leq\Delta e(\gamma_t)+\lambda\, e(\gamma_t).
\end{equation}

\emph{(2')} If  $|Z|_{L^\infty(M,E)}, |\nabla Z|_{L^\infty(M,E)}<\infty$, then
\begin{equation}
\dnd[\kappa(\gamma_t)]{t}\leq\Delta \kappa(\gamma_t)+4|R^M\!|e(\varphi_t)\kappa(\varphi_t)+\lambda\,\kappa(\gamma_t)+\mu\, e(\gamma_t)^{1/2}\kappa(\gamma_t),
\end{equation} where $\lambda=\lambda(M,Z)=\frac{1}{2}|Z|^2_{L^\infty(M,E)}$ and $\mu=\mu(M, \nabla Z)=2^{3/2}|\nabla Z|_{L^\infty(M,E)}$ are constants only depending on $M, Z$ and $\nabla Z$.  All metrics and norms used here are the natural ones induced by the metrics $g$ and $G$.
\end{cor}

\section{Short time existence}\label{stime}

Now, let us carry out step 1) of our program and show the short time existence of solutions to the IVP (\ref{iniprob}). To this end, we cast the parabolic initial value problem in a form that is analytically easier to handle with.
As before let $(\Sigma^k,g)$ and $(M^n,G)$ be Riemannian manifolds, and let $\Sigma$ be compact and oriented. 
Furthermore, let $Z$ be  a smooth section of $\Hom(\Lambda^{\!  k} TM, TM)\cong\Lambda^{\!  k} T^*\! M\otimes TM$ and $f\in C^\infty(\Sigma,M)$ be the initial condition from (\ref{iniprob}).
We use  Nash's imbedding theorem, which says that any Riemannian manifold can be isometrically imbedded into an Euclidean space of sufficient high dimension, in order to isometrically imbedd $M$ into a certain $\R^q$. Let \[ \iota : M\hookrightarrow\R^q\] denote the isometric imbedding, and let $\tilde{M}$ be a tubular neighborhood of the submanifold $\iota(M)\subset\R^q$. It can be defined as an open subset of $\R^q$ by
\[ \tilde{M}=\{(x,v)\;|\; x\in\iota(M),\; v\in T_x\iota(M)^\perp,\; |v|<\epsilon(x)\}.\] Here, $\epsilon : M\to (0,\infty)$ is a positive smooth function on $M$. By
\[\pi:\tilde{M}\to\iota(M)\]
we denote the canonical projection which assigns to each $z\in\tilde{M}$ the closest point in $\iota(M)$ from $z$. We extend this projection to a smooth map $\pi:\R^q\to\R^q$  that vanishes outside $\tilde{M}$. This can be done by choosing the positive function $\epsilon$ small enough. Also the bundle homomorphism $Z$ can be extended to a bundle homomorphism  $\tilde{Z}:\Lambda^{\!k} T\R^q\to T\R^q$, meaning that $d\iota\circ Z=\tilde{Z}\circ (d\iota)^{\underline{k}}$ holds; and we do this as follows: Denote by $\tilde{M}_1,\tilde{M}_2$ smaller tubular neighborhoods of $M$ such that $M\subset\tilde{M}_1\subset\tilde{M}_2\subset\tilde{M}$ holds. For example, as  $\tilde{M}_1$ and $\tilde{M}_2$ we can take the $\epsilon/4$-tubular neighborhood and the $\epsilon/2$-tubular neighborhood, respectively, both contained in the above defined $\epsilon$-tubular neighborhood $\tilde{M}$. In $\tilde{M}_2$ we define $\tilde{Z}$ by
\[\tilde{Z}_x(\xi):=d\iota\Big(Z_{\pi(x)}((d\pi)^{\underline{k}}(\xi))\Big),\] for all $\xi\in\Lambda^{\!k} T_x\R^q$ and all $x\in \tilde{M}_2$. Here, we have identified all  tangent spaces $T_x\R^q\cong T_y\R^q\cong\R^q$ by parallel translation.  Then  choose a smooth function $\psi:\R^q\to\R$ with support in $\tilde{M}_2$ such that $\psi\equiv 1$ in the closure of $\tilde{M}_1$ and $0\leq \psi\leq 1$ in $\R^q$ hold. Multiplying the above $\tilde{Z}$ defined in $\tilde{M}_2$ by this cut-off function $\psi$, yields a smooth bundle map $\tilde{Z}:\Lambda^{\!k}T\R^q\to T\R^q$ which is globally defined in $\R^q$ and vanishes outside $\tilde{M}_2$.\\

Now,
let $u:\Sigma\times[0,T)\to\tilde{M}$ be a map from $\Sigma\times[0,T)$ into $\tilde{M}\subset\R^q$.
Regarding $u$ as a function with values in $\R^q$, we may consider the following initial value problem (IVP) for the system of parabolic partial differential equations:\rbox{iniprobl}
\begin{equation}\label{iniprobl}
\left\{\begin{array}{lc }
(\Delta-\dnd{t})\,u(x,t)=\Pi_{u}(du,du)(x,t)+\tilde{Z}_u((du)^{\underline{k}})(x,t),\quad\quad
(x,t)\in\Sigma\times (0,T),\\
u(x,0)=\iota\circ f(x).     
\end{array}\right.
\end{equation} Here, $\Delta=-\delta d$ is the Hodge Laplacian of $\Sigma$ componentwise applied to $u$ and $f$ is the map given as initial condition of the IVP (\ref{iniprob}). $\tilde{Z}$ is the extension
of the $k$-force as described above and $\Pi(du,du)$ is a vector in $\R^q$ defined as follows. Let $\{e_i\}$ be a local orthonormal frame field on $\Sigma$ regarded, by canonically extension, as a local frame field on $\Sigma\times(0,T)$. Then
\begin{equation}
\Pi(du,du):=\tr\nabla d\pi(du,du)=(\nabla_{du(e_i)} d\pi)(du(e_i)).
\end{equation} 
 We consider only those solutions $u :\Sigma\times [0,T)\to\tilde{M}$ to the IVP (\ref{iniprobl}) which are continuous on $\Sigma\times[0,T)$, $C^2$ differentiable in $\Sigma$ and of class $C^1$ in $(0,T)$. In symbols this means
\[u\in C^0(\Sigma\times [0,T), \tilde{M})\cap C^{2,1}(\Sigma\times (0,T), \tilde{M}).\] The relation between the two initial value problems is ruled by the following. \rbox{iniequi}

\begin{prop}\label{iniequi}
Let $u\in C^0(\Sigma\times [0,T), \tilde{M})\cap C^{2,1}(\Sigma\times (0,T), \tilde{M})$. If $u$ is a solution to the initial value problem $\mathrm{(\ref{iniprobl})}$, then $u(\Sigma\times [0,T))\subset\iota(M)$ holds true and $\varphi=\iota^{-1}\circ u$ is a solution to the  \emph{IVP (\ref{iniprob})}. The converse also holds true.
\end{prop}
\begin{proof}
Suppose that  $u\in C^0(\Sigma\times [0,T), \tilde{M})\cap C^{2,1}(\Sigma\times (0,T), \tilde{M})$ is a solution to the IVP (\ref{iniprobl}) and let $\tilde{Z}$ be the extension of $Z\in\Gamma(\mathrm{Hom}(\Lambda^{\!k} TM,TM))$ constructed above. At first we will show that $u(\Sigma\times [0,T))\subset\iota(M)$ holds. For this we define a map $\rho:\tilde{M}\to\R^q$ by
\[\rho(z)=z-\pi(z),\quad z\in \tilde{M}, \] and a function $h:\Sigma\times [0,T)\to\R^q$ by
\[h(x,t)=|\rho(u(x,t))|^2,\quad (x,t)\in\Sigma\times [0,T).\] We see, by definition, that
$\rho(z)=0$ iff $z\in\iota(M)$. Thus, we only have to verify $h\equiv 0$. Since $u(x,0)=\iota(f(x))\in\iota(M)$, we see $h(x,0)=0$. As $u$ is a solution to the IVP (\ref{iniprobl}), we obtain with $\rho(u)=\rho\circ u$ 
\begin{align*}
\dnd[h]{t} &= \dnd{t} \langle\rho(u),\rho(u)\rangle=
2\big\langle d\rho\big(\dnd[u]{t}\big),\rho(u)\big\rangle\\
&=2\big\langle d\rho(\Delta u-\Pi(du,du)-Z((du)^{\underline{k}})),\rho(u)\big\rangle,
\end{align*}
\begin{align*}
\Delta h &=\Delta\langle\rho(u),\rho(u)\rangle\\
&=2\langle\Delta\rho(u),\rho(u)\rangle+2|d\rho(u)|^2,
\end{align*} where $\langle\;,\,\rangle$ is the scalar product in $\R^q$. The formula for the second fundamental form of composite maps (see Lemma \ref{zweite} below) says
\[\Delta\rho(u)=d\rho(\Delta u)+\tr\nabla d\rho(du,du), \] where $\Delta$ is the Hodge-Laplacian of $\Sigma$. Since, by definition, $\pi(z)+\rho(z)=z$, we have $d\pi+d\rho=id$ and $\nabla d\pi+\nabla d\rho=0$. This together with the fact that the images of $d\pi$ and $\rho$ are orthogonal to each other yields
\begin{align*}
\Delta h &= 2\langle d\rho(\Delta u)-\tr\nabla d\pi(du,du),\rho(u)\rangle+2|d\rho(u)|^2\\
&= 2\langle d\rho(\Delta u-\Pi(du,du)),\rho(u)\rangle+2|d\rho(u)|^2,
\end{align*} and hence, 
\begin{align}\label{crucialesti}\dnd[h]{t} &=\Delta h - 2|d\rho(u)|^2-2\langle d\rho(\tilde{Z}((du)^{\underline{k}})),\rho(u)\rangle\notag\\
&=\Delta h - 2|d\rho(u)|^2-2\langle \tilde{Z}((du)^{\underline{k}}),\rho(u)\rangle\\
&=\Delta h - 2|d\rho(u)|^2.\notag
\end{align} The term $\langle \tilde{Z}((du)^{\underline{k}}),\rho(u)\rangle$ vanishes since $\tilde{Z}((du)^{\underline{k}})\perp\rho(u)$ by construction of  $\tilde{Z}$.
 Then by the Divergence Theorem  we have for each $t\in (0,T)$,
\begin{align*}
\ddt\int_\Sigma h(\cdot,t)\,d\vol_g &=\int_\Sigma \dnd[h]{t}(\cdot,t)\,d\vol_g= -2\int_\Sigma |d\rho(u)|^2\,d\vol_g\leq 0. \end{align*}          Since $h(x,0)=0$ from the assumption, we have
\[\int_\Sigma h(\cdot,t)\,d\vol_g\leq\int_\Sigma h(\cdot,0)\,d\vol_g=0\] and consequently
 $h\equiv 0$. \\
 
 Now, we turn to the second half of the assertion. Therefore, let $u : \Sigma\times [0,T)\to \tilde{M}$ be a solution to the IVP (\ref{iniprobl}). From the previous assertion we know that $u(\Sigma\times [0,T))\subset\iota(M)$. Hence, we can write $u=\iota\circ \varphi$, where $\varphi$ is a map from $\Sigma\times [0,T)$ to $M$. We will show that $\varphi$ is a solution to the IVP (\ref{iniprob}). Due to the formula (see Lemma \ref{zweite}) for the second fundamental form of composition maps for $u=\iota\circ \varphi$ and for $\iota=\pi\circ\iota$ we get

\begin{align*}
\Delta u&=\tr\nabla d\iota (d\varphi,d\varphi)+d\iota(\tau(\varphi)),\\
\nabla d\iota &=\nabla d\pi(d\iota,d\iota)+d\pi(\nabla d\iota).
\end{align*} Since $\iota : M\to \R^q$ is an isometric imbedding, the second fundamental $\nabla d\iota$ of $\iota$ is orthogonal to $\iota(M)$ at each point, and thus $d\pi(\nabla d\iota)=0$. Combining this and the preceding equations, we obtain
\[d\iota(\tau(\varphi))=\Delta u-\tr\nabla d\pi(du,du).\]  Bearing in mind that $d\iota\circ Z=\tilde{Z}\circ (d\iota)^{\underline{k}}$ and $d\iota(\dnd[\varphi]{t})=\dnd[u]{t}$ hold, we  finally arrive at
\[d\iota\Big(\tau(\varphi)-\dnd[\varphi]{t}-Z((d\varphi)^{\underline{k}})\Big)=(\Delta-\dnd{t})\,u-\tilde{Z}((du)^{\underline{k}})-\Pi(du,du).\] From this one  reads off that $\varphi$ is a solution to the IVP (\ref{iniprob}) if $u$ is a solution to the initial value problem (\ref{iniprobl}). Analogously the converse can  easily be verified.
\end{proof}

In the proof of the preceding proposition we have made use of the following lemma which can be verified by a simple calculation.

\begin{lemma}\label{zweite}
Let $(\Sigma,g)$, $(M,G)$ and $(N,h)$ be Riemannian manifolds. Given maps $\Sigma\overset{\varphi}{\longrightarrow} M \overset{\psi}{\longrightarrow} N$, we have $\nabla d(\psi\circ \varphi)=d\psi(\nabla d\varphi)+\nabla d\psi(d\varphi,d\varphi)$; and $\tau(\psi\circ \varphi)=d\psi(\tau(\varphi))+\tr\nabla d\psi (d\varphi,d\varphi)$.
\end{lemma}

From Proposition \ref{iniequi} we see that we can prove short time existence for solutions to the IVP (\ref{iniprob}) by establishing short time existence for IVP (\ref{iniprobl}). For the latter IVP one can set up a function space which is well adapted to our problem. To this end,  we follow Lady$\check{\mathrm{z}}$enskaya, Solonnikov and Ural'ceva (\cite{lsu}, p. 7). Given $T>0$, set $Q=\Sigma\times [0,T]$. Let $0<\alpha<1$. Given a vector valued function $u:Q\to\R^q$, set
\[|u|_Q=\underset{(x,t)\in Q}{\mathrm{sup}} |u(x,t)|,\]
\[\langle u\rangle_x^{(\alpha)}=\underset{x\neq x'}{\underset{(x,t),(x',t)\in Q}{\mathrm{sup}}}
\frac{|u(x,t)-u(x',t)|}{d(x,x')^\alpha},\]
\[\langle u\rangle_t^{(\alpha)}=\underset{t\neq t'}{\underset{(x,t),(x,t')\in Q}{\mathrm{sup}}}
\frac{|u(x,t)-u(x,t')|}{|t-t'|^\alpha},\] and define the norms $|u|_Q^{(\alpha,\alpha/2)}$, $|u|_Q^{(2+\alpha,1+\alpha/2)}$ by
\begin{align}\label{norms}|u|_Q^{(\alpha,\alpha/2)}&=|u|_Q+\langle u\rangle_x^{(\alpha)}+\langle u\rangle_t^{(\alpha/2)},\notag\\
|u|_Q^{(2+\alpha,1+\alpha/2)} &= |u|_Q+|\partial_t u|_Q+|D_x u|_Q+|D^2_x u|_Q \\
 & \mspace{23mu}+\langle \partial_t u\rangle_t^{(\alpha/2)}+\langle D_x u\rangle_t^{(1/2+\alpha/2)}+\langle D^2_x u\rangle_t^{(\alpha/2)}\notag\\
 & \mspace{23mu}+\langle \partial_t u\rangle_x^{(\alpha)}+\langle D_x^2 u\rangle_x^{(\alpha)}.\notag
\end{align} Here, $d(x,x')$ is the Riemannian distance between $x$ and $x'$ in $\Sigma$ and $\partial_t u$ represents $\partial u/\partial t$. Also $D_x u$ and $D^2_x u$ represent the first order derivative of $u$ in $\Sigma$ and its covariant derivative, respectively. In terms of a local coordinate system $(x^i)$ in $\Sigma$ and the standard coordinates $(y^\alpha)$ of $\R^q$, $D_x u$ and $D^2_x u$ are, respectively, given by
\[D_x u= du=\partial_i u^\alpha\cdot dx^i\otimes\dnd{y^\alpha},\]
\[D^2_x u= \nabla du=\nabla_i\partial_j u^\alpha\cdot dx^i\otimes dx^j\otimes\dnd{y^\alpha},\]  and $| D_x u|^2_Q$ and $|D^2_x u|^2_Q$ are, respectively, given as

\[| D_x u|^2_Q=\underset{(x,t)\in Q}{\mathrm{sup}}g^{ij}\partial_i u^\alpha\partial_j u^\alpha,\]
\[| D^2_x u|^2_Q=\underset{(x,t)\in Q}{\mathrm{sup}}g^{ik}g^{jl}\nabla_i\partial_j u^\alpha \nabla_k\partial_l u^\alpha,\] where $\partial_i=\partial/\partial x^i$. With respect to these norms we define the function spaces $C^{\alpha,\alpha/2}(Q,\R^q)$ and $C^{2+\alpha,1+\alpha/2}(Q,\R^q)$, respectively, by
\[C^{\alpha,\alpha/2}(Q,\R^q)=\{ u\in C^0(\Sigma\times [0,T])\;|\; |u|_Q^{(\alpha,\alpha/2)}<\infty\},\]
\[C^{2+\alpha,1+\alpha/2}(Q,\R^q)=\{ u\in C^{2,1}(\Sigma\times [0,T])\;|\; |u|_Q^{(2+\alpha,1+\alpha/2)}<\infty\},\] and set
\[C^{2+\alpha,1+\alpha/2}(Q,M)=\{ u\in C^{2+\alpha,1+\alpha/2}(Q,\R^q)\;|\; u(Q)\subset M\},\] where we have naturally identified $M$ with $\iota(M)\subset \R^q$.  One can show that $C^{\alpha,\alpha/2}(Q,\R^q)$ and  $C^{2+\alpha,1+\alpha/2}(Q,\R^q)$ are Banach spaces with norms
$|u|_Q^{(\alpha,\alpha/2)}$, $|u|_Q^{(2+\alpha,1+\alpha/2)}$, respectively. They are called \textit{H\"older spaces} on $Q\times [0,T]$. See \cite{evans}, \cite{gilbarg} for example. $C^{2+\alpha,1+\alpha/2}(Q,M)$ is a closed subset of $C^{2+\alpha,1+\alpha/2}(Q,\R^q)$. This follows immediately because $M$, as a compact subset,  is closed in $\R^q$.\\

Now, we prove the following.  \rbox{shortime}

\begin{theo}\label{shortime}
Let $(\Sigma,g)$ and $(M,G)$ be  Riemannian manifolds, and $\Sigma$ be compact and oriented. Furthermore, let $Z\in\Gamma(\Hom(\Lambda^{\! k}TM,TM))$. For any $C^{2+\alpha}$ map $f\in C^{2+\alpha}(\Sigma,M)$ there exists a positive number $\epsilon=\epsilon(\Sigma,M,Z,f,\alpha)>0$ and a map $u\in C^{2+\alpha,1+\alpha/2}(\Sigma\times [0,\epsilon], \tilde{M})$ such that $u$ is a solution in $\Sigma\times [0,\epsilon)$ to the  \textup{IVP (\ref{iniprobl})}. Here, $\epsilon=\epsilon(\Sigma,M,Z,f,\alpha)$ is a constant depending on $\Sigma, M,Z, f,$ and $\alpha$.
\end{theo} The main tool that we use to prove this theorem is the Inverse Function Theorem for Banach spaces. It says that a $C^1$ map is locally invertible at a point iff its linearization is invertible at this point. The idea is to apply the Inverse Function Theorem to reduce the solvability of a nonlinear differential equation to the solvability of its linearized version. However, before it we review the following classically well known result about existence and uniqueness of solutions to linear parabolic partial differential equations. (see \cite{lsu}, p. 320) or (\cite{evans}, p. 350 ff.) 
\rbox{linear}
\begin{theo}\label{linear}
Let $(\Sigma,g)$ be a compact Riemannian manifold of dimension $k$, and set $Q=\Sigma\times [0,T]$. Given a vector valued function $u:Q \to \R^q$, let
\[Lu=\Delta u+\mathbf{a}\cdot\nabla u+\mathbf{b}\cdot u-\partial_t u\]
be a linear parabolic partial differential operator, and consider the initial value problem \rbox{classical}
\begin{equation}\label{classical}
\left\{\begin{array}{lc }
Lu(x,t)=F(x,t),\quad\quad
(x,t)\in\Sigma\times (0,T),\\
u(x,0)=f(x).
\end{array}\right.
\end{equation} Here, the components of $\Delta u$, $\mathbf{a}\cdot\nabla $, $\mathbf{b}\cdot u$, $\partial_t u$ are, respectively, defined by
\[\Delta u^A,\quad \mathbf{a}^{iA}_B(x,t)\dnd[u^B]{x^i},\quad \mathbf{b}^A_B(x,t) u^B, \quad  \dnd[u^A]{t}, \quad 1\leq A \leq q.\] If
\[\mathbf{a}^{iA}_B,\; \mathbf{b}^A_B\in C^{\alpha,\alpha/2}(Q,\R), \quad 1\leq i\leq k,\quad 1\leq A,B\leq q,\] for some $0<\alpha <1$, then for any 
\[F\in C^{\alpha,\alpha/2}(Q,\R^q), \quad f\in C^{2+\alpha}(\Sigma,\R^q),\] there exist a unique solution
$u\in C^{2+\alpha,1+\alpha/2}(Q,\R^q)$ to \textup{(\ref{classical})} such that
\[|u|_Q^{(2+\alpha,1+\alpha/2)}\leq C(|F|_Q^{(\alpha,\alpha/2)}+|f|_\Sigma^{(2+\alpha)})\] holds. Here, the constant $C=C(\Sigma, L, q, T, \alpha)$ only depends on $\Sigma, L, q, T, \alpha$.
\end{theo} 

Now, we turn to the proof of Theorem \ref{shortime}.

\begin{proof} At first let $\tilde{Z}$ be the smooth extension of $Z$ constructed at the beginning of this section.
We choose an $\alpha'$ such that $0<\alpha'<\alpha<1$ and use the abbreviation $\partial_t=\partial/\partial t$.\\

\emph{Step} 1 (Construction of an approximate solution). Consider the following initial value problem of a system of linear parabolic partial differential equations: \rbox{approx}
\begin{equation}\label{approx}
\left\{\begin{array}{lc }
(\Delta-\dnd{t})\, v(x,t)=\Pi_{f}(df,df)(x,t)+\tilde{Z}_f((df)^{\underline{k}})(x,t),\quad\quad
(x,t)\in\Sigma\times (0,1),\\
v(x,0)=f(x),     
\end{array}\right.
\end{equation} where we have identified  $f$ with $\iota\circ f$. From the assumption $f\in C^{2+\alpha}(\Sigma,\R^q)$ we get \[\Pi_f(df,df),\,\tilde{Z}_f((df)^{\underline{k}})\in C^\alpha(\Sigma,\R^q)\subset C^{\alpha,\alpha/2}(\Sigma\times [0,1],\R^q),\] and consequently by virtue of the previous Theorem \ref{linear} the existence of a unique solution
\[v\in C^{2+\alpha,1+\alpha/2}(\Sigma\times [0,1],\R^q)\] to the IVP (\ref{approx}). If we denote the desired solution by $u$, then $v$ approximates $u$ at $t=0$ in the following sense,
\[v(x,0)=u(x,0),\quad\quad \partial_t v(x,0)=\partial_t u(x,0).\]

\emph{Step} 2 (Application of the Inverse Function Theorem). Now, putting $Q=\Sigma\times [0,1]$,  we consider the differential operator
\[P(u)=\Delta u-\partial_t u-\Pi_u(du,du)-\tilde{Z}_u((du)^{\underline{k}})\] and note that an $u\in C^{2+\alpha,1+\alpha/2}(\Sigma\times [0,\epsilon],\R^q)$ satisfying $P(u)=0$ is our desired solution.\\

For $0<\alpha'<1$ we introduce the subspaces $X$ and $Y$ in $C^{2+\alpha',1+\alpha'/2}(Q,\R^q)$ and $C^{\alpha',\alpha'/2}(Q,\R^q)$, respectively, by
\begin{align*}X &=\{h\in C^{2+\alpha',1+\alpha'/2}(Q,\R^q)\; |\; h(x,0)=0,\; \partial_t h(x,0)=0\},\\
Y &= \{k\in C^{\alpha',\alpha'/2}(Q,\R^q)\; |\; k(x,0)=0\}.
\end{align*} The spaces $X$ and $Y$ are, by definition, closed subspaces; and  hence Banach spaces. We define a map $\mathcal{P} :X \to Y$ by 
\[\mathcal{P}(h)=P(v+h)-P(v),\quad \mbox{for }\; h\in X.\] From the definition of $P$ and $X$ we see that $\mathcal{P}(h)\in C^{\alpha',\alpha'/2}(Q,\R^q)$ and $\mathcal{P}(h)(x,0)=0$ for $h\in X$ so that in fact $\mathcal{P}(h)\in Y$ holds true. In particular, $\mathcal{P}(0)=0$. $\mathcal{P}$ is Fr$\acute{\mathrm{e}}$chet differentiable in a neighborhood of $h=0$. A direct computation using the definition of $\mathcal{P}$  shows that the Fr$\acute{\mathrm{e}}$chet derivative $\mathcal{P}'(0) :X \to Y$, for $h\in X$, is given by

\begin{align*}\mathcal{P}'(0)(h) &= \Delta h-\partial_{t}h-(d\Pi)\big|_v(h)(dv,dv)-2\Pi_v(dv,dh)\\
&\mspace{23mu}-(d\tilde{Z})\big|_v(h)((dv)^{\underline{k}})
-\tilde{Z}_v(dh\tilde{\wedge} (dv)^{\underline{k-1}}).
\end{align*} Here, 
$\tilde{Z}(dh\tilde{\wedge} (dv)^{\underline{k-1}})=\tilde{Z}((dh\tilde{\wedge} (dv)^{\underline{k-1}})(\vol_g^\sharp))$ and $(d\tilde{Z})(h)((dv)^{\underline{k}})= (d\tilde{Z})(h)((dv)^{\underline{k}}(\vol_g^\sharp))$, respectively. (For the definition of the $\tilde{\wedge}$-product, see Appendix \ref{appA}(a).)        From this it can readily be  verified that $\mathcal{P}'(0):X\to  Y$ is an isomorphism of Banach spaces. In fact, since $v\in C^{2+\alpha,1+\alpha/2}(Q,\R^q)$, from the definition of $\mathcal{P}'(0)$ and Theorem \ref{linear} we see that for any $K\in Y$ there exists a unique $H\in C^{2+\alpha',1+\alpha'/2}(Q,\R^q)$ satisfying 
\begin{equation*}
\left\{\begin{array}{lc }
\mathcal{P}'(0)(H)(x,t)=K(x,t),\quad\quad
(x,t)\in\Sigma\times (0,1),\\
H(x,0)=0.
\end{array}\right.
\end{equation*} We also see that for such a $H$ the following estimate holds:
\begin{equation}\label{inject}
|H|_Q^{(2+\alpha',1+\alpha'/2)}\leq C|K|_Q^{(\alpha',\alpha'/2)}.
\end{equation}  Since $K(x,0)=0$ and $H(x,0)=0$ hold, we obtain  $\partial_t H(x,0)=0$; and thus $H\in X$. From this and the definition of $X$, $Y$ and the expression for $\mathcal{P}'(0)$  we know that $\mathcal{P}'(0)$ is a bounded and surjective linear mapping of Banach spaces. Equation (\ref{inject}) tells us that $\mathcal{P}'(0)$ is injective and the Open Mapping Theorem from functional analysis that also the inverse $\mathcal{P}'(0)^{-1}$ is  bounded. Hence, $\mathcal{P}'(0)$ is an isomorphism. \\

Applying the Inverse Function Theorem for Banach spaces, $\mathcal{P}: X\to Y$ is  a homeomorphism between a sufficiently small neighborhood $\mathcal{U}$ of $0\in X$ and a neighborhood $\mathcal{P}(\mathcal{U})$ of $0\in Y$. This means that we can find a positive number $\delta=\delta(\Sigma, M, Z,f)>0$, depending only on $\Sigma,M,Z$ and $f$, such that the following holds: For any $k\in C^{\alpha',\alpha'/2}(Q,\R^q)$ with $k(x,0)=0$ and $|k|_Q^{(\alpha',\alpha'/2)}<\delta$, there exists a  $h\in C^{2+\alpha',1+\alpha'/2}(Q,\R^q)$ satisfying  
\begin{equation}\label{conditions}
\mathcal{P}(h)=k,\quad h(x,0)=0,\quad \partial_t h(x,0)=0.
\end{equation} Here, $\delta=\delta(\Sigma,M,Z,f)$ is a positive number determined by $\Sigma,M,Z$ and $f$. Setting $u=v+h$ and $w=P(v)$, from (\ref{conditions}) we see that there exists a 
$u\in C^{2+\alpha',1+\alpha'/2}(Q,\R^q)$ satisfying
\begin{equation}\label{intermed}
\left\{\begin{array}{lc }
P(u)(x,t)=(w+k)(x,t),\quad\quad
(x,t)\in\Sigma\times (0,1),\\
u(x,0)=f(x).
\end{array}\right.
\end{equation}
 
\emph{Step} 3 (Short time existence). For a given real number $\epsilon>0$ consider a $C^\infty$ function 
$\zeta:\R\to\R$ satisfying $\zeta(t)=1\;(t\leq\epsilon)$, $\zeta(t)=0\; (t\geq 2\epsilon)$, $0\leq\zeta(t)\leq 1$, $|\zeta'(t)|\leq2/\epsilon\;(t\in\R)$. We note that $w=P(v)\in C^{\alpha,\alpha/2}(Q,\R^q)\subset C^{\alpha',\alpha'/2}(Q,\R^q)$ and that $w(x,0)=0$ holds from the definition of $P(v)$, $v\in C^{2+\alpha,1+\alpha/2}(\Sigma\times [0,1],\R^q)$ and $v(x,0)=f(x)$. By a straightforward computation we see that there exist a constant $C>0$ independent of $\epsilon$ and $w$ such that the estimate
\begin{equation}\label{estimate}
|\zeta w|_Q^{(\alpha',\alpha'/2)}\leq C\epsilon^{(\alpha-\alpha')}|w|_Q^{(\alpha,\alpha/2)}
\end{equation} holds. Set $k=-\zeta w$. Then $k(x,0)=0$. From (\ref{estimate}) we have
$|k|_Q^{(\alpha',\alpha'/2)}<\delta$ for sufficiently small $\epsilon$. Thus, there exists a $u\in C^{2+\alpha',1+\alpha'/2}(\Sigma\times [0,\epsilon],\R^q)$ such that the following special case of (\ref{intermed}) holds:
\begin{equation*}
\left\{\begin{array}{lc }
P(u)(x,t)=0,\quad\quad
(x,t)\in\Sigma\times (0,\epsilon),\\
u(x,0)=f(x).
\end{array}\right.
\end{equation*} In other words, we have obtained a solution $u\in C^{2+\alpha',1+\alpha'/2}(\Sigma\times [0,\epsilon],\R^q)$ to the initial value problem
\begin{equation*}
\left\{\begin{array}{lc }
(\Delta-\partial_t)\, u(x,t)=\Pi_u(du,du)(x,t)+\tilde{Z}_u((du)^{\underline{k}})(x,t),\quad\quad
(x,t)\in\Sigma\times (0,\epsilon),\\
u(x,0)=f(x).
\end{array}\right.
\end{equation*} As we have
\[f\in C^{2+\alpha}(\Sigma,\R^q),\quad \Pi_u(du,du),\,\tilde{Z}_u((du)^{\underline{k}})\in C^{\alpha,\alpha/2}(\Sigma\times [0,\epsilon],\R^q),\]     we see by Theorem \ref{linear} that
\[u\in C^{2+\alpha,1+\alpha/2}(\Sigma\times [0,\epsilon],\R^q).\] Due to compactness of $\Sigma$ and  continuity of $u$ we always can reach that $u(\Sigma\times [0,\epsilon'])\subset\tilde{M}$ holds true if we choose $0<\epsilon'<\epsilon$ small enough . Replacing $\epsilon$ by $\epsilon'$ if necessary, we may assume that  $u(\Sigma\times [0,\epsilon])\subset\tilde{M}$ holds true. Thus, $u$ is a solution to the IVP (\ref{iniprobl}) in $\Sigma\times [0,\epsilon]$. It is also clear from the above proof that $\epsilon>0$ is a positive number only depending on $\Sigma, M,Z, f$ and $\alpha$.
\end{proof}

As a result of combining Proposition \ref{iniequi} and Theorem \ref{shortime}, we obtain the following. \rbox{shorty}

\begin{cor}\label{shorty} Let $(\Sigma,g)$ and $(M,G)$ be Riemannian manifolds, and $\Sigma$ be compact and oriented. Furthermore, let  $Z\in\Gamma(\Hom(\Lambda^{\! k}TM,TM))$. For a given  $C^{2+\alpha}$ map $f\in C^{2+\alpha}(\Sigma,M)$ there exist  a positive number $T=T(\Sigma, M,Z, f, \alpha)>0$ and  a map $\varphi\in C^{2+\alpha,1+\alpha}(\Sigma\times [0,T],M)$ such that
\begin{equation}
\left\{\begin{array}{lc }
\tau(\varphi_t)(x)=Z((d\varphi_t)^{\underline{k}})(x)+\dnd[\varphi_t]{t}{(x)},\quad\quad
(x,t)\in\Sigma\times (0,T),\\
\varphi(x,0)=f(x)
\end{array}\right.
\end{equation}  holds. Here, $T=T(\Sigma, M, Z,f, \alpha)>0$ is a constant depending on $\Sigma, M,Z, f$ and $\alpha$ alone.
\end{cor} 

From regularity theory for  solutions to  linear parabolic partial differential equations, we obtain the following (see Appendix \ref{appB}, Theorem \ref{ellireg}).

\begin{theo}[Short time existence]\label{shor} Let $(\Sigma,g)$ and $(M,G)$ be  Riemannian manifolds, and $\Sigma$ be compact  and oriented. Furthermore, let  $Z\in\Gamma(\Hom(\Lambda^{\! k}TM,TM))$. For a given $C^{2+\alpha}$ map $f\in C^{2+\alpha}(\Sigma,M)$ there exist a positive number $T=T(\Sigma, M, Z,f, \alpha)>0$ and
a map $\varphi\in C^{2+\alpha,1+\alpha/2}(\Sigma\times [0,T],M)\cap C^{\infty}(\Sigma\times (0,T),M)$ such that
\begin{equation}\label{para}
\left\{\begin{array}{lc }
\tau(\varphi_t)(x)=Z((d\varphi_t)^{\underline{k}})(x)+\dnd[\varphi_t]{t}{(x)},\quad\quad
(x,t)\in\Sigma\times (0,T),\\
\varphi(x,0)=f(x)
\end{array}\right.
\end{equation} holds. Here, $T=T(\Sigma, M,Z, f, \alpha)>0$ is a constant depending on $\Sigma, M,Z, f$ and $\alpha$ alone.
\end{theo}

\section{Long time existence}  \label{ltime}

To prove long time existence of a solution $\varphi:\Sigma\times [0,T)\to M$ to the initial value problem (IVP) for the system of nonlinear parabolic partial differential
equations 
\rbox{inipro}
\begin{equation}\label{inipro}
\left\{\begin{array}{lc }
\tau(\varphi_t)(x)=Z((d\varphi_t)^{\underline{k}})(x)+\dnd[\varphi_t]{t}{(x)},\quad\quad
(x,t)\in\Sigma\times (0,T),\\
\varphi(x,0)=f(x),
\end{array}\right.
\end{equation} one has to show that it exists when $T=\infty$. Short time existence of a solution to (\ref{inipro}) can be guaranteed   by Theorem \ref{shor} in contrast to long time existence. As already mentioned in Section \ref{gmethod} it becomes an essential matter to control the growth rate of the solution $\varphi(x,t)$  in time $t$. In order to get a grip on the "blowing up" effects of the nonlinear terms of the equation, the dimension of $\Sigma$ and the compactness of $M$ plays a crucial role in this game. In fact, in $\dim(\Sigma)>1$ the nonlinear terms possibly may destroy the long time behavior of our solutions. 
The main ingredients are the energy estimates and the maximum principle for parabolic equations. Both are typical tools in the theory of linear partial differential equations to get a priori estimates that allow to show e.g. uniqueness  and stability  of solutions. For an introduction to this topic see \cite{evans}, \cite{renrog}. Here, we state a version of the maximum principle that will suffice our needs. A proof can be found in \cite{nish}, p. 142. \\
\begin{lemma}[Maximum principle] Let $(\Sigma,g)$ be a compact Riemannian manifold. Furthermore, let $\Delta$ be the Hodge-Laplacian of $\Sigma$ and $L=\Delta-\dnd{t}$ be the heat operator. Let $u\in C^0(\Sigma\times [0,T))\cap C^{2,1}(\Sigma\times (0,T))$ be a real valued function in $\Sigma\times [0,T)$, which is $C^2$ in $\Sigma$ and $C^1$ in $(0,T)$. If $u$ satisfies $Lu\geq 0$ in $\Sigma\times (0,T)$, then
\[\max_{\Sigma\times [0,T)}u=\max_{\Sigma\times\{0\}}u\] holds. Said in words, the maximum of $u$ in the "cylinder" $\Sigma\times [0,T)$ is achieved at the bottom of the cylinder, i.e. in $\Sigma\times\{0\}$.
\end{lemma}

In the sequel  we denote by $S^1$ the unit circle in $\R^2$, carrying the induced metric by $\R^2$. Set $E=\mathrm{Hom}(\Lambda^{\! k}TM,TM)$.
From Corollary \ref{energyestimates} in Section \ref{gmethod} and the maximum principle we gain the following estimates for a solution to the IVP (\ref{inipro}).\rbox{engyesti}

\begin{prop}[Energy estimates]\label{engyesti} Assume that $\Sigma= S^1$ and $Z$ is a Lorentz force. Let $\varphi=\gamma\in C^{2,1}(S^1\times [0,T),M)\cap C^\infty(S^1\times (0,T),M)$ be a solution to the \emph{IVP} $\mathrm{(\ref{inipro})}$ and set $\gamma_t(s)=\gamma(s,t)$.  Then the following  hold: \\

\emph{(1)} If $|Z|_{L^\infty(M,E)}<\infty$, then for all $(s,t)\in S^1\times [0,T)$,

\begin{align*}
e(\gamma_t)(s)\leq e^{\lambda T}\,\underset{s\in\Sigma}{\sup}\;e(f)(s).
\end{align*}

\emph{(2)} If $\underset{M}{\sup}|R^M\!|< \infty$ and $|Z|_{L^\infty(M,E)}, |\nabla Z|_{L^\infty(M,E)}<\infty$, then for all $(s,t)\in S^1\times [0,T)$,
\begin{align*}
\Big|\dnd[\gamma]{t}(s,t)\Big|\leq e^{CT}\,\underset{s\in\Sigma}{\sup}\;\Big|\dnd[\gamma]{t}(s,0)\Big|.
\end{align*} Here, $\lambda=\lambda(M,Z)$ and $\mu=\mu(M,\nabla Z)$ are the constants defined in \emph{Corollary \ref{energyestimates}} and $C=C(\Sigma,M,Z,\nabla Z, f,T)=4\,\underset{M}{\sup}|R^M\!|\,e^{\lambda T}\,\underset{\Sigma}{\max}\;e(f)+\lambda+\mu\,e^{\frac{\lambda T}{2}} \underset{\Sigma}{\max}\;e(f)^{1/2}$ is a constant depending on $\Sigma,M,Z,\nabla Z,f$ and $T$ alone. 
\end{prop}

\begin{proof} ad (1): From (1') of Corollary \ref{energyestimates} we see
\[L e(\gamma_t)=\Big(\Delta-\dnd{t}\Big)\,e(\gamma_t)\geq-\lambda e(\gamma_t).\]
Putting $v(s,t)= e^{-\lambda t} \,e(\gamma_t)(s)$, a straightforward computation shows that $v$ satisfies
 $Lv\geq 0$ in $S^1\times (0,T)$. Hence, from the maximum principle and the definition of the energy density $e(\gamma_t)$
 
  \[e^{-\lambda t}\,e(\gamma_t)(s)=v(s,t)\leq \underset{s\in S^1}{\max}\;v(s,0)=\underset{s\in S^1}{\max}\;e(f)(s)\] holds at any $(s,t)\in S^1\times [0,T)$.\\
  
  ad (2): Let $C$ be the constant defined as above. From (1) of Proposition \ref{engyesti} and (2') of Corollary \ref{energyestimates} we see that for $v(s,t):= e^{-C t} \,\kappa(\gamma_t)(s)$, we have $L\kappa(\gamma_t)\geq 0$ in $S^1\times (0,T)$.  Hence, from the maximum principle and the definition of the energy density $\kappa(\gamma_t)$
 
 \[e^{-C t}\Big|\dnd[\gamma]{t}(s,t)\Big|^2=2\,v(s,t)\leq 2\,\underset{s\in S^1}{\max}\;v(s,0)=\underset{s\in S^1}{\max}\;\Big|\dnd[\gamma]{t}(s,0)\Big|^2\]
 holds at any $(s,t)\in S^1\times [0,T)$.
\end{proof}

Proposition \ref{engyesti} implies that the growth rate of a solution $\gamma$ to the initial value problem (\ref{inipro}) is uniformly bounded on $S^1\times [0,T)$ with respect to the time variable $t\in [0,T)$, if $\underset{M}{\sup}|R^M\!|< \infty$ and  $|Z|_{L^\infty(M,E)}, |\nabla Z|_{L^\infty(M,E)}<\infty$. More precisely we state the following.

\begin{prop}\label{fundesti}Assume that $\Sigma=S^1$. Furthermore let  $(M,G)$ be a compact Riemannian manifold and $\varphi=\gamma\in C^{2,1}(S^1\times [0,T),M)\cap C^\infty(S^1\times (0,T),M)$ be a solution to the \emph{IVP} $\mathrm{(\ref{inipro})}$. Set $\gamma_t(s)=\gamma(s,t)$.  Let $Z$ be a Lorentz force. Then for any $0<\alpha<1$ there exists a positive number $C=C(\Sigma,M,Z,f,\alpha,T)>0$ such that
\[  |\gamma(\cdot,t)|_{C^{2+\alpha}(S^1,M)}+\Big|\dnd[\gamma]{t}(\cdot,t)\Big|_{C^{\alpha}(S^1,M)}\leq C\]
holds at any $t\in [0,T)$. Here, $C=C(\Sigma,M,Z,\nabla Z, f,\alpha,T)$ is a constant only depending on $\Sigma, M,Z,\nabla Z, f,\alpha$ and $T$. 
\end{prop}

\begin{proof} We set $\gamma'_t(s)=\dnd[\gamma]{s}(s,t)$. All metrics and norms here are the natural induced ones. As in the proof of Proposition \ref{shortime}, we assume the $(M,G)$ is realized as a Riemannian submanifold in a $q$-dimensional Euclidean space $\R^q$ via an isometric imbedding $\iota:M\hookrightarrow\R^q$ and that the vector valued function $\gamma: S^1\times [0,T)\to\R^q$ is a solution to the IVP (\ref{iniprobl}). Furthermore, let $\tilde{Z}$ be the smooth extension of $Z$, constructed at the beginning of Section \ref{stime}. However, since $\gamma$, from the assumption, is a solution to the IVP (\ref{inipro}), the solution stays inside $M\subset\R^q$ and therefore all expressions, terms and constants $c_i$, appearing in the course of the proof will only depend on $Z$ and its covariant derivatives, but not on $\tilde{Z}$ and its covariant derivatives. Thus, for simplicity we denote $\tilde{Z}$ by $Z$.  \\

Now, depending on the point of view,  $\gamma$ satisfies an elliptic and, on the other hand, a parabolic partial differential equation. We will exploit both  positions in order to attain our result. Taking the first view, $\gamma$ satisfies the system of elliptic partial differential equations
\[\Delta\gamma=\Pi_\gamma(d\gamma,d\gamma)+Z_\gamma(d\gamma)+\dnd[\gamma]{t},\] where $\Delta$ is the Hodge-Laplacian in $\Sigma$. Noting Proposition \ref{engyesti}, we see that the right hand side of the above equation is bounded independent of $t\in [0,T)$, i.e. we have
\begin{equation}\label{estiA}
\Big|\Pi_\gamma(d\gamma,d\gamma)(\cdot,t)+Z_\gamma(d\gamma)(\cdot,t)+\dnd[\gamma]{t}(\cdot,t)\Big|_{L^\infty(S^1,\R^q)}\leq c_1(\Sigma,M,Z,f,T).
\end{equation} In fact, for all $(s,t)\in S^1\times [0,T)$ we have
\begin{eqnarray*}
\Big|\Pi_\gamma(d\gamma,d\gamma)(s,t)+Z_\gamma(d\gamma)(\cdot,t)+\dnd[\gamma]{t}(s,t)\Big|=
\Big|(\nabla_{\gamma'_t}d\pi)(\gamma'_t)(s)+Z_\gamma(\gamma'_t)(s)+
\dnd[\gamma_t]{t}(s)\Big|\\
\leq|\nabla d\pi|_{L^\infty(M,E)}\,|\gamma'_t(s)|^2+\frac{1}{2}|Z|^2_{L^\infty(M,E)}+\frac{1}{2}|\gamma'_t(s)|^2+\Big|\dnd[\gamma_t]{t}(s)\Big|.
\end{eqnarray*}
The right hand side of this inequality can be estimated from above by Proposition \ref{engyesti} with a constant
$c_1$ only depending on $\Sigma,M,Z,\nabla Z, f$ and $T$ (actually $c_1$ also depends on $|\nabla d\pi|_{L^\infty(M,E)}$, but we won't pick this up in our notation). Here, $|Z|_{L^\infty(M,E)}=\sup_M\langle \nabla Z,\nabla Z\rangle^{1/2}$ and $|\nabla d\pi|_{L^\infty(M,E)}=\sup_M\langle \nabla d\pi,\nabla d\pi\rangle^{1/2}$. This shows (\ref{estiA}). \\

Since the image of $\gamma$ is always contained in the bounded set $ M\subset\R^q$, at any $t\in [0,T)$ we have
\begin{equation}\label{bounded}|\gamma(\cdot,t)|_{L^\infty(S^1,\R^q)}\leq c_2(M).\end{equation}   Hence, by the Schauder estimate (see Appendix \ref{appB}, Theorem \ref{ellireg}) for the solutions to an elliptic partial differential equation, at any $t\in [0,T)$ we have
\begin{align}\label{estiB}
|\gamma(\cdot,t)|_{C^{1+\alpha}(S^1,\R^q)}&\leq c_3(\Sigma,\alpha)
\Big(\;\underset{t\in [0,T)}{\sup}\;|\Delta\gamma(\cdot,t)|_{L^\infty(S^1,\R^q)}+\notag
\underset{t\in[0,T)}{\sup}\;|\gamma(\cdot,t)|_{L^\infty(S^1,\R^q)}\;\Big)\\
&\leq c_4(\Sigma,M,Z,\nabla Z,f,\alpha,T).
\end{align} Taking the second view, $\gamma$ is also a solution to the system of parabolic partial differential equations
\[L\gamma=\Pi_\gamma(d\gamma,d\gamma)+Z_\gamma(d\gamma),\] where $L=\Delta-\dnd{t}$ is the heat operator in $S^1$.  Regarding (\ref{estiB})  we see that
\[|\Pi_\gamma(d\gamma,d\gamma)(\cdot,t)+Z_\gamma(d\gamma)(\cdot,t)|_{C^{\alpha}(S^1,\R^q)}\leq c_5(\Sigma,M,Z,\nabla Z, f,\alpha,T)\] holds. Using the Schauder estimate for linear parabolic partial differential equations (see Appendix \ref{appB}, Theorem \ref{ellireg}) , we get for any $t\in [0,T)$

\begin{align*}
|\gamma(\cdot,t)|_{C^{2+\alpha}(S^1,\R^q)}+\Big|\dnd[\gamma]{t}(\cdot,t)\Big|_{C^{\alpha}(S^1,\R^q)}\\
& \mspace{-175mu}\leq c_6(\Sigma,\alpha)\Big (\;\underset{t\in [0,T)}{\sup}|L\gamma(\cdot,t)|_{C^\alpha(S^1,\R^q)}+\underset{t\in [0,T)}{\sup}|\gamma(\cdot,t)|_{L^\infty(S^1,\R^q)}\Big)\\
&\mspace{-175mu}\leq c_7(\Sigma,M,Z,\nabla Z, f,\alpha,T).
\end{align*} \end{proof}

Now, we are ready to proof the main theorems.\\

\textbf{Proof of Theorem 1.}
Short time existence is guaranteed by Theorem \ref{shor}, namely there exists a positive number $T=T(\Sigma,M,Z,f,\alpha)>0$ such that, without making any curvature assumptions, the initial value problem (\ref{pama}) has a solution $\gamma\in C^{2+\alpha,1+\alpha/2}(S^1\times [0,T],M)\cap C^{\infty}(S^1\times (0,T),M)$ in $S^1\times [0,T]$. We have to demonstrate now that our solution can not blow up in finite time if $M$ is compact, i.e. that our solution $\gamma$ can be extended to $S^1\times [0,\infty)$. Setting
\[T_0=\sup\,\{ t\in [0,\infty)\,|\, (\ref{pama}) \mbox{ has a solution in } S^1\times [0,t]\}, \]
we must show that $T_0=\infty$ holds. Assume that this would not be the case. Then choose any sequence of numbers $\{t_i\}\subset [0,T_0)$ such that $t_i\to T_0$ as $i$ tends to $\infty$. As in the proof of Proposition \ref{fundesti} we regard $M$ to be an isometrically imbedded submanifold in some Euclidean space $\R^q$ and each $\gamma(\cdot,t_i)\in C^\infty(S^1,M)$ as a $\R^q$-valued function. We set $\gamma_t(s)=\gamma(s,t)$, $\gamma'=\gamma'_t=\dnd[\gamma]{s}$,  $\partial_t=\dnd{t}$ and choose a positive number $\alpha'$ such that $0<\alpha<\alpha'<1$. Since $S^1$ is compact, it follows that the imbedding $C^{k+\alpha'}(S^1,\R^q)\hookrightarrow C^{k+\alpha}(S^1,\R^q)$ is compact.
By Proposition \ref{fundesti} the sequences \[\{\gamma(\cdot,t_i)\}\;\; \mbox{ and }\;\;\{\partial_t\gamma(\cdot,t_i)\},\] respectively, are bounded in $C^{2+\alpha'}(S^1,\R^q)$ and in $C^{\alpha'}(S^1,\R^q)$. Thus, there exist a subsequence $\{t_{i_k}\}$ of $\{t_{i}\}$  and functions
\[\gamma(\cdot,T_0)\in C^{2+\alpha}(S^1,\R^q)\;\; \mbox{ and }\;\;\partial_t\gamma(\cdot,T_0)\in C^{\alpha}(S^1,\R^q)\] such that the subsequences 
\[\{\gamma(\cdot,t_{i_k})\}\;\; \mbox{ and }\;\;\{\partial_t\gamma(\cdot,t_{i_k})\},\] respectively, converge uniformly to $\gamma(\cdot,T_0)$ and  $\partial_t\gamma(\cdot,T_0)$, as $t_{i_k}\to T_0$. Since for each $t_{i_k}$ we have
\[\partial_t\gamma(\cdot,t_{i_k})=\frac{\nabla }{\partial s}\gamma'(\cdot,t_{i_k})-Z(\gamma')(\cdot,t_{i_k}),\] we also get at $T_0$
\[\partial_t\gamma(\cdot,T_0)=\frac{\nabla }{\partial s}\gamma'(\cdot,T_0)-Z(\gamma')(\cdot,T_0).\]
Consequently,  we see that (\ref{pama}) has a solution in $S^1\times [0,T_0]$. Application of Theorem \ref{shor} with $\gamma(\cdot,T_0)$ as initial value, yields an positive number $\epsilon>0$ such that the IVP 
\begin{equation}
\left\{\begin{array}{lc }
\frac{\nabla }{\partial s}\gamma'_t(s)=Z(\gamma'_t)(s)+\dnd[\gamma_t]{t}{(s)},\quad\quad
(s,t)\in S^1\times (T_0,T_0+\epsilon),\\
\gamma(s,0)=\gamma(s,T_0)
\end{array}\right.
\end{equation} has a solution $\gamma\in C^{2+\alpha,1+\alpha/2}(S^1\times [T_0,T_0+\epsilon],M)$ in $S^1\times [T_0,T_0+\epsilon]$. Noting that this and the previous solution coincide on $S^1\times \{0\}$, we can patch them together to a solution $\gamma \in C^{2+\alpha,1+\alpha/2}(S^1\times [0,T_0+\epsilon],M)$ to the IVP (\ref{pama}). From the arguments concerning the differentiability of the solutions in Theorem \ref{shor} we see that $\gamma$ is $C^\infty$ in $S^1\times (0,T_0+\epsilon)$. Hence, (\ref{pama}) has a solution in $S^1\times [0,T_0+\epsilon]$ which contradicts the definition of $T_0$. Consequently $T_0=\infty$. The uniqueness of $\gamma$ immediately follows from Theorem \ref{unique}. \hspace{0.4cm}$\square$\\

\textbf{Proof of Theorem 2.}
As in the proof of Proposition \ref{shortime} we regard $u,v$ as vector valued functions
$u,v: S^1\times [0,T)\to\iota(M)\subset\R^q$, and consider $u,v$ as solutions to the system of nonlinear parabolic differential equations (\ref{iniprobl}). Let $\tilde{Z}$ and $\tilde{Z}'$ be the smooth extensions of $Z$ and $Z'$, respectively, constructed 	as at the beginning of Section \ref{stime}. However, since the solution must stay in $M\cong\iota(M)\subset\R^q$, the majority of  appearing expressions, involving $\tilde{Z}$ and $\tilde{Z}'$, only depend on $Z$ and $Z'$. 
Define  a function $h:\Sigma\times [0,T)\to\R$ by
\[h(s,t)=|u(s,t)-v(s,t)|^2,\quad (s,t)\in S^1\times [0,T).\]
For $u_1,u_2\in C^2(S^1,\R^q)$, one computes
\[\Delta\langle u_1,u_2\rangle=\langle\Delta u_1,u_2\rangle+2\langle du_1,du_2\rangle+\langle u_1,\Delta u_2\rangle,\]
and hence for $u_1=u_2=u-v$ we get
\[\Delta h=\Delta\langle u-v,u-v\rangle=2\langle\Delta u-\Delta v, u-v\rangle+2|du-dv|^2.\]
On the other hand, one has
\[\dnd[h]{t}=2\langle\Delta u-\Delta v-\big(\Pi_u(du,du)-\Pi_{v}(dv,dv)+Z_u(du)-Z'_{v}(dv)\big),u-v\rangle.\]
Then for $L=\Delta-\dnd{t}$ it follows 
\begin{align}\label{eins}
Lh&=\langle\Pi_u(du,du)-\Pi_{v}(dv,dv),u-v\rangle+\langle Z_u(du)-Z'_{v}(dv),u-v\rangle\\
&\mspace{23mu}+2|du-dv|^2.\notag
\end{align}
Now, for $0<T_0<T$ we choose a number $r=r(T_0)$ such that $u(S^1\times [0,T_0])\cup v(S^1\times [0,T_0])$ is contained in the open ball $B(0,r)=\{x\in\R^q\, | \, |x|<r\}$. Rewriting
\begin{align*}Z_u(du)-Z'_{v}(dv)
=(Z_u-Z_{v})(du)+(Z_{v}-Z'_{v})(du)+Z'_{v}(du-dv)\end{align*} and applying the Mean Value Theorem to $Z_u-Z_{v}$, we get for any $(s,t)\in S^1\times [0,T_0]$
\begin{align}\label{zwei}
|\langle Z_u(du)-Z_v(dv),u-v\rangle|\\
&\mspace{-80mu}\leq c_1\,|u-v|^2+2^{1/2}|Z-Z'|_{L^\infty(M,E)}\,e(u_t)^{1/2}|u-v|\notag\\
&\mspace{-59mu}+c_3\,|du-dv||u-v|\notag\\
&\mspace{-80mu}\leq c_1\,|u-v|^2+|Z-Z'|^2_{L^\infty(M,E)}+c_2\,|u-v|^2\notag\\
&\mspace{-59mu}+c_3\,|du-dv||u-v|\notag.
\end{align} Here,  $c_1,c_2,c_3\geq 0$ are nonnegative constants. $c_1$ only depends on $\Sigma, M, \nabla \tilde{Z}, T_0$ and on the maximum value of the energy density $e(u_t)$ on $\Sigma\times [0,T_0]$,  $c_2$ only  on the maximum value of the energy density $e(u_t)$ on $\Sigma\times [0,T_0]$, whereas $c_3$ only depends on $\overline{B(0,r)}$ and $Z'$, i.e. on $T_0$ and  $Z'$. Note that the energy densities can be globally estimated independent of $T_0$ by virtue of Proposition \ref{engyesti}. In fact, noting $\sup_{\overline{B(0,r)}}\,|\nabla \tilde{Z}|<\infty$  (here  $|\nabla \tilde{Z}|=\langle\nabla \tilde{Z},\nabla \tilde{Z}\rangle^{1/2}$ as usual)
 and applying the Mean Value Theorem  yields Lipschitz continuity, namely \[|\tilde{Z}_x(\xi)-\tilde{Z}_y(\xi)|\leq (\sup_{\overline{B(0,r)}}|\nabla \tilde{Z}|)|\xi||x-y|\] holds, for all $x,y\in \overline{B(0,r)}\subset\R^q$ and all $\xi\in\Lambda^{\! k}\R^q$. Here, we have identified $\Lambda^{\! k}T_x\R^q\cong\Lambda^{\! k}T_y\R^q\cong\Lambda^{\! k}\R^q$  by parallel transport. From this, (\ref{zwei}) can readily be verified.  Similarly rewriting
\begin{align*}
\Pi_u(du,du)-\Pi_{v}(dv,dv)\\
&\mspace{-80mu}=(\Pi_u-\Pi_{v})(du,du)+\Pi_{v}(du-dv,du)+\Pi_{v}(dv,du-dv)
\end{align*} and applying the Mean Value Theorem to $\Pi_u-\Pi_{v}$, we get for any $(s,t)\in S^1\times [0,T_0]$
\begin{align} \label{drei}
|\langle\Pi_u(du,du)-\Pi_{v}(dv,dv),u-v\rangle|\\
&\mspace{-80mu}\leq c_4\,|u-v|^2+c_5\,|du-dv||u-v|,\notag
\end{align} where $c_4, c_5\geq 0$ are constants only depending on $\Sigma, M$, on the maximum values of the energy densities $e(u_t)$ and $e(v_t)$ on $S^1\times [0,T_0]$, and on derivatives of the canonical projection $\pi:\tilde{M}\to M$ up to third order.
 Using Cauchy's inequality $ab\leq\epsilon a^2+(4\epsilon)^{-1}b^2$ $(a,b\geq 0,\;\epsilon>0)$ for the terms \[\mbox{\emph{constant}}\cdot|du-dv||u-v|,\]  we obtain from (\ref{eins}), (\ref{zwei}) and (\ref{drei}) for any $(s,t)\in S^1\times [0,T_0]$
\begin{align*}\label{eins}
Lh &\geq -|\langle\Pi_u(du,du)-\Pi_{v}(dv,dv),u-v\rangle|-|\langle Z_u(du)-Z'_{v}(dv),u-v\rangle|\\
&\mspace{23mu}+2|du-dv|^2\\
&\geq -C|u-v|^2-|Z-Z'|^2_{L^\infty(M,E)}=-Ch-|Z-Z'|^2_{L^\infty(M,E)},
\end{align*} where $C\geq 0$ is a constant only depending on $\Sigma, M, Z,\nabla \tilde{Z}, Z',T_0$,  on the maximum values of the energy densities $e(u_t)$ and $e(v_t)$ on $S^1\times [0,T_0]$, and on derivatives of the canonical projection $\pi:\tilde{M}\to M$ up to third order.  Integrating and using  the Divergence Theorem  yields for any $t\in [0,T_0]$
\[\ddt\int_{\Sigma} h(\cdot,t)\,d\vol_g\leq C\int_{\Sigma} h(\cdot,t)\,d\vol_g + 2\pi|Z-Z'|^2_{L^\infty(M,E)}.\] Here,  $g$ denotes the canonical metric of $\Sigma=S^1\subset\R^2$ induced by $\R^2$. Applying Gronwall's Lemma  to the function $H:[0,T)\to \R$ defined by $H(t)=\int_{\Sigma} h(\cdot,t)\,d\vol_g$, we get  for any $t\in [0,T_0]$
\[H(t)\leq e^{Ct}\big(H(0)+2\pi t|Z-Z'|^2_{L^\infty(M,E)}\big).\] This together with $H(0)\leq 2\pi|u_0-v_0|^2_{L^\infty(\Sigma,M)}$ yields the desired estimate.\hspace{2.63cm} $\square$\\

\textbf{Conclusion and outlook.} We see that the energy estimates (Corollary \ref{enestimates}) are crucial to make the "long time existence proof\," work. If $k=\dim(\Sigma)=1$, the maximum principle can be applied to obtain good a priori estimates for the energy densities. Even in the case $k>1$, the maximum principle is not applicable and the proof breaks down. The greater $k>1$ is, the worse the nonlinearities become. Perhaps in $\dim(\Sigma)=2$, where the nonlinearities are "only" of quadratic order in $du$, i.e. $|Z((du)^{\underline{k}})|\leq C |du|^k$ ($C>0$ a constant) for a bounded $k$-force $Z$, existence of weak long time solution can be shown. It would be an interesting task to prove the existence of long time solutions in this case especially regarding the relevance of this question in String theory. Also an open question is the third item of program presented in Section \ref{gmethod}: Does one always find a convergent subsequence of a long time solution to the IVP (\ref{pama}) when $(M,G)$ is compact Riemannian manifold?

\begin{center}\textbf{Acknowledgement}\end{center}
{\footnotesize As this work is a result of my thesis I would like to pronounce here some credits.
Firstly I would like to thank my supervisor Christian B\"ar. He has been extremely helpful, insightful  and encouraging at all times and it has been a great pleasure to work with him. I would also like to thank Gerhard Huisken for the enlightening discussions which have given me a lot of insights. Thanks are also due to my room mate Florian Hanisch for many interesting and useful conversations and to my colleagues from our differential geometry team. Finally I would also like to acknowledge the support of the Sonderforschungsbereich Raum-Zeit-Materie (SFB 647) of the DFG.}

\begin{appendix}  \section{Notation and definitions}\label{appA}

\textbf{(a)} \textbf{Geometric notation.} Let $(\Sigma,g)$ and $(M,G)$ be Riemannian manifolds and $(E,\nabla^E,h)$ be a Riemannian vector bundle over $M$. For simplicity we denote the metrics $g,G,h$ and all the induced metrics and connections on the various tensor bundles   by $\langle\cdot,\cdot\rangle$ and $\nabla$, respectively. If $\Sigma$ is oriented, then we denote by $\vol_g$ the canonical volume form on $\Sigma$ (similarly for $M$). If $\Sigma$  is not orientable, then in expressions
\[\int_\Sigma f\,d\vol_g,\] where $f:\Sigma\to \R$ is an integrable function, the symbol $d\vol_g$ is to mean the Riemannian measure which can be defined for any Riemannian manifold.  The characteristic function of a measurable set $A\subset\Sigma$ is denoted by $\chi_A$ and its volume by $V(A)=\int_\Sigma \chi_A\,d\vol_g$. 

\begin{dfn}  A \emph{Riemannian vector bundle} over $M$ is a triple $(E,\nabla^E,h)$ consisting of a real vector bundle $E$ and a connection $\nabla^E$ on $E$ that is compatible with the metric $h$ of $E$, i.e. $\nabla^E_X h=0$ for all $X\in T_x M$, for all $x\in M$. \end{dfn}
This induces connections $\nabla$ on the bundles $\Lambda^{\!  k} T^*\! M\otimes E$ for $k=1,\ldots,\dim(M)$.
The induced curvature for $\omega\in\Gamma(\Lambda^{\!  k} T^*\! M\otimes E)$ is defined by
\begin{equation}
R(X_1,X_2)\omega=\{\nabla_{X_1}\nabla_{X_2}-\nabla_{X_2}\nabla_{X_1}-\nabla_{[X_1,X_2]}\}\omega.
\end{equation}

As for real-valued $k$-forms one can define an exterior differential $d$ and a co-differential $\delta$ for forms with values in bundles (see \cite{xin}).
The \emph{Hodge-Laplace operator} $\Delta:\Gamma(\Lambda^{\!  k} T^*\! M\otimes E)\to\Gamma(\Lambda^{\!  k} T^*\! M\otimes E)$ then is given by
     \begin{equation}
     \Delta =-\{d\delta+\delta d\},
     \end{equation} and the \emph{rough Laplacian} 
     $\bar{\Delta}:\Gamma(\Lambda^{\!  k} T^*\! M\otimes E)\to\Gamma(\Lambda^{\!  k} T^*\! M\otimes E)$ by
     \begin{equation}
     \bar{\Delta}\omega=\{\nabla_{e_i}\nabla_{e_i}-\nabla_{\nabla_{e_i}e_i}\}\omega
     \end{equation}     
     
Furthermore on $\Gamma(\Lambda^{\! k} T^{*}\! M\otimes E)$  we use the following convention for the induced metric. Let $\{e_i\}$ be an orthonormal frame near $x\in M$, then for $\alpha,\beta\in\Gamma(\Lambda^{\! k} T^{*} \! M\otimes E)$ we define
\[\langle\alpha,\beta\rangle_{\wedge}:=\sum_{i_1<\cdots< i_k}\langle\alpha(e_{i_1},\ldots,e_{i_k}),\beta(e_{i_1},\ldots,e_{i_k})\rangle.\] We distinguish this metric from that naturally induced metric for non totally skew-symmetric $k$-linear vector valued tensor fields $\alpha,\beta\in\Gamma(\bigotimes^k T^*\! M\otimes E)$ which is given by
\[ \langle\alpha,\beta\rangle=\langle\alpha(e_{i_1},\ldots,e_{i_k}),\beta(e_{i_1},\ldots,e_{i_k})\rangle.\]
Note that this two definitions are related by a factor $1/k!$, namely for $\alpha,\beta\in\Gamma(\Lambda^{\! k} T^{*} \! M\otimes E)\subset\Gamma(\bigotimes^k T^*\! M\otimes E)$,  we have
\[ \langle\alpha,\beta\rangle_{\wedge}=\frac{1}{k!}\langle\alpha,\beta\rangle.\] In this paper we supress the subscript $\wedge$ with the convention that $\langle\alpha,\beta\rangle$ is to mean $\langle\alpha,\beta\rangle_{\wedge}$ if $\alpha,\beta\in\Gamma(\Lambda^{\! k} T^{*} \! M\otimes E)$. In particular, for the volume element we have $|\vol_G|=\langle\vol_G ,\vol_G\rangle^{1/2}=1$ due to this convention.\\

We recall some notions from Linear Algebra.
Let $(V,g)$ and $(W,h)$ be a Euclidean vector spaces. The
isomorphism $V\to V^*, \, \xi\mapsto
g(\xi,\cdot)$ is denoted by $\xi^\flat$ for $\xi\in V$
and its inverse by $\omega^{\sharp}$ for $\omega\in V^*$.
One can extend these isomorphisms  to 
$\Lambda^{\!  k} V$ and $\Lambda^{\!  k} V^*$. On
decomposable $k$-vectors it is defined by
$(\xi_1\wedge\ldots\wedge\xi_k)^{\flat}:=\xi_1^\flat\wedge\ldots\wedge\xi_k^\flat$
and extended by linearity; similarly
 $(\omega_1\wedge\ldots\wedge\omega_k)^{\sharp}:=\omega_1^\sharp\wedge\ldots\wedge\omega_k^\sharp$ on
 decomposable $k$-forms. Here, we use the convention
\begin{equation*} 
(\omega_1\wedge\ldots\wedge\omega_k)(\ovec{\xi})=\sum_{\sigma\in S_k}(-1)^{\sigma}\omega_1(\xi_{\sigma(1)})\cdots \omega_k(\xi_{\sigma(k)}),
\end{equation*}     where $S_k$ denotes the permutation group of order $k$, i.e. $\sigma$ runs over all $k$-permutations. The sign  $(-1)^\sigma$ of the permutation equals $+1$ if the permutation $\sigma$ is even and $-1$ if it is odd.
More general,    By the universal property of the exterior product this induces a linear map $A_1\tilde{\wedge}\ldots\tilde{\wedge} A_k :\Lambda^{\!  k} V\to \Lambda^{\!  k} W$, denoted by the same symbol, such that on decomposable $k$-vectors, we have \begin{equation*}(A_1\tilde{\wedge}\ldots\tilde{\wedge} A_k)(\xi_1\wedge\ldots\wedge\xi_k)= \sum_{\sigma\in S_k}(-1)^{\sigma} A_1(\xi_{\sigma(1)})\wedge\ldots\wedge A_k(\xi_{\sigma(k)}).
\end{equation*} For a single linear map $A:V\to W$ we define a  linear map  $A^{\underline{k}}:\Lambda^{\!  k} V\to \Lambda^{\!  k} W $  by
\begin{equation}\label{fakultaet}
A^{\underline{k}}:=\frac{1}{k!}\,A^k,
\end{equation} where $A^k$ denotes the $k$-fold $\tilde{\wedge}$-product of $A$ with itself,
\[A^k=\underbrace{A\tilde{\wedge}\ldots\tilde{\wedge} A}_{\mbox{\footnotesize{k-times}}}.\] Note that for $\ovec{\xi}\in V$ with $|\xi_i|\leq 1$ ($i=1,\ldots, k$), we have $|A^{\underline{k}}(\xi_1\wedge\ldots\wedge\xi_k)|\leq |A|^k$. Here, $|A|=\big[\sum_i\left\langle A(e_i),A(e_i)\right\rangle\big]^{1/2}$ for any orthonormal basis $\{e_i\}$ of $V$.
 Let   $X=\mathrm{Hom}(V,W)$ denote the vector space of endomorphisms from $V$ to $W$ and set $\tilde{\Lambda}^k X=\mathrm{Hom}(\Lambda^{\!  k} V,\Lambda^{\!  k} W)$. There is a natural product $\tilde{\Lambda}^k X\otimes\tilde{\Lambda}^l X\to \tilde{\Lambda}^{k+l}X$ given by
\[(A_1\tilde{\wedge}\ldots\tilde{\wedge} A_k)\otimes (A_{k+1}\tilde{\wedge}\ldots\tilde{\wedge} A_{k+l})\mapsto A_{1}\tilde{\wedge}\ldots\tilde{\wedge} A_{k+l}\] which is associative and symmetric. Note that in general $\tilde{\Lambda}^k X\neq\Lambda^{\!  k} X$, e.g. for $A\in X, A \neq 0$ we have $A\tilde{\wedge}  A\neq 0$ in $\tilde{\Lambda}^2 X,$ but $A\wedge A=0$ in $\Lambda^{\!  2} X$. Let  $\ovec{A}: E\to F$ be bundle homomorphisms,  $(E,\nabla^E)$ and $(F,\nabla^F)$ be bundles with connection over a Riemannian manifold $(M,G)$ and $\eta,\ovec{\xi}\in\Gamma(TM)$. Then we define a connection $\tilde{\nabla}$ on $\tilde{\Lambda}^k X$ (here $X=\mathrm{Hom}(E,F)$) by
\begin{eqnarray*}\lefteqn{\big(\tilde{\nabla}_\eta(A_1\tilde{\wedge}\ldots\tilde{\wedge} A_k)\big)(\xi_1\wedge\ldots\wedge\xi_k)  :=}\\
&&\nabla_\eta(A_1\tilde{\wedge}\ldots\tilde{\wedge} A_k)(\xi_1\wedge\ldots\wedge\xi_k)
-(A_1\tilde{\wedge}\ldots\tilde{\wedge} A_k)\big(\nabla_\eta(\xi_1\wedge\ldots\wedge\xi_k)\big).\end{eqnarray*} For convenience we have denoted  the natural induced connections  on $\Lambda^{\!  k} E$ and $\Lambda^{\!  k} F$, respectively, simply by $\nabla$. It follows immediately that the Leibniz rule is satisfied, i.e. \begin{eqnarray*}\tilde{\nabla}_\eta(A_1\tilde{\wedge}\ldots\tilde{\wedge} A_k)= \tilde{\nabla}_\eta A_1\tilde{\wedge} A_2\tilde{\wedge}\ldots\tilde{\wedge}A_k+\cdots+A_1\tilde{\wedge} A_2\tilde{\wedge}\ldots\tilde{\wedge}A_{k-1}\tilde{\wedge}\tilde{\nabla}_\eta A_k.\end{eqnarray*}

\textbf{(b)} \textbf{Function spaces.} Let $(\Sigma,g)$ and $(M,G)$ be compact Riemannian manifolds and $(E,\nabla^E,h)$ be a Riemannian vector bundle over $M$. As usual we denote the continuous, the $k$-times continuous differentiable and the smooth functions from $\Sigma$ to $M$ by $C^0(\Sigma,M)$, $C^k(\Sigma,M)$ and $C^\infty(\Sigma,M)$, respectively. The smooth sections in $E$ with basis $M$ are denoted by $\Gamma(M,E)$. If the reference to the base space is clear, we just write $\Gamma(E)=\Gamma(M,E)$. For $M=\R$ and $0\leq k\leq \infty$ we set $C^k(\Sigma)=C^k(\Sigma,\R)$. Let $|\cdot|$ denote the norm induced by the h of $E$. Then for $1\leq p\leq\infty$ the $L^p$-spaces $L^p(M,E)$ are defined as measurable sections in $E$ with finite norm $|s|_{L^p(M,E)}<\infty$. Here, for $p=\infty$, we put $|s|_{L^\infty(M,E)}=\inf\{r\in\R\; |\; |u|\leq r \mbox{ holds a.e.}\}$ and for $1\leq p<\infty$ \[|s|_{L^p(M,E)}=\left(\int_M |s|^p\,d\vol_G\right)^{1/p}.  \]

If $E=M\times \R^q$ is the trivial  bundle with canonical metric and trivial connection over $M$, for $1\leq p\leq \infty$ we set $L^p(M,\R^q)=L^p(M,M\times\R^q)$  and especially for $q=1$ we  write $L^p(M)=L^p(M,\R)$. By $L^p(\Sigma,M)$ we mean the space $\{u\in L^p(\Sigma,\R^q)\; |\; u(\Sigma)\subset M \}$, where  $M\subset\R^q$ is regarded as an isometrically imbedded submanifold in some Euclidean space $\R^q$. Let $0<\alpha<1$ be a positive real number, $k$ be a nonnegative integer, and $U\subset\R^n$ be an open subset in $\R^n$.  Then the H\"older spaces are denoted by $C^\alpha(U)$ and for $k\geq 1$ by $C^{k+\alpha}(U)$, respectively. It is well-known that one can define in a similar way a H\"older norm and H\"older spaces $C^{k+\alpha}(M)$ on a  Riemannian manifold $M$ by means of parallel translation (see \cite{joy}, Chapter 1).
For a vector valued function $u:M\to \R^q$,  we say that $u$ belongs to $C^{k+\alpha}(M,\R^q)$ if all its components $u^i$ belong to $C^{k+\alpha}(M).$ Finally, by  $C^{k+\alpha}(\Sigma,M)$ we mean the space $\{u\in C^{k+\alpha}(\Sigma,\R^q)\; |\; u(\Sigma)\subset M \}$, where  $M\subset\R^q$ is regarded as an isometrically imbedded submanifold in some Euclidean space $\R^q$.

\section{Analytical toolbox}\label{appB}

Given $r>0$, set $B(0,r)=\{x\in\R^n\; |\; |x|<r\}.$  Let $P$ be a linear elliptic partial differential operator given by
\[ P=\sum_{i,j=1}^n a^{ij}(x)\frac{\partial^2}{\partial x^i\partial x^j}+\sum_{i=1}^n b^i(x)\dnd{x^i} +d(x). \] 
Assume that
 that $P$ is uniformly elliptic, i.e. that
\[\lambda|\xi|^2\leq \sum_{i,j=1}^n a^{ij}(x)\xi^i\xi^j\leq\Lambda|\xi|^2\] holds for some constants $0<\lambda\leq\Lambda<\infty$ and for any $x\in B(0,r)$ and $\xi\in\R^n$. Given $T>0$, set $Q=B(0,r)\times (0,T)$. For a function $u:Q\to \R$, we set

\[\langle u\rangle_x^{(\alpha)}=\underset{x\neq x'}{\underset{(x,t),(x',t)\in Q}{\sup}}\frac{|u(x,t)-u(x',t)|}{|x-x'|^\alpha},\]
\[\langle u\rangle_t^{(\alpha)}=\underset{t\neq t'}{\underset{(x,t),(x,t')\in Q}{\sup}}\frac{|u(x,t)-u(x,t')|}{|t-t'|^\alpha}.\]
The norms $|u|_Q^{(\alpha,\alpha/2)}$ and $|u|_Q^{(2+\alpha,1+\alpha/2)}$ are defined as  (\ref{norms}) in Section \ref{stime}. By $C^{\alpha,\alpha/2}(Q), \,C^{2+\alpha,1+\alpha/2}(Q)$ we denote the H\"older spaces with respect to these norms. We then have the following.

\begin{theo}\label{ellireg} 
\emph{\textbf { a) Differentiability of solutions}}\\
\emph{(1)}  Given $0<\alpha<1$, assume that $a^{ij},b^i,d,f\in C^\alpha(B(0,r))$. Then $u\in C^{2+\alpha}(B(0,r))$ holds if $u\in C^2(B(0,r))$ satisfies the linear partial differential equation  \[Pu(x)=f(x) \tag{$\star$}.\]
Furthermore, if $a^{ij},b^i,d,f\in C^{k+\alpha}(B(0,r))$ for a given  $k\geq 1$, then a solution $u$ to \emph{($\star$)} is $C^{k+2+\alpha}$. In particular, if $a^{ij},b^i,d,f\in C^{\infty}(B(0,r))$, then $u\in C^{\infty}(B(0,r))$.\\

\emph{(2)}  Given $0<\alpha<1$, assume that $a^{ij},b^i,d\in C^\alpha(B(0,r))$ and $f\in C^{\alpha,\alpha/2}(Q)$. Then $u\in C^{2+\alpha,1+\alpha/2}(Q)$ holds, if $u\in C^{2,1}(Q)$ satisfies the following linear parabolic partial differential equation  \[\Big(P-\dnd{t}\Big)\,u(x,t)=f(x,t) \tag{$\star\star$}.\]\\
Furthermore, let $p,q$ be nonnegative integers. Given $\beta, \kappa$ with $|\beta|\leq p,\,|\beta|+2\kappa\leq p,\,\kappa\leq q$, assume  that $D^\beta_x a^{ij},D^\beta_x b^i,D^\beta_x d\in C^{\alpha}(B(0,r))$  and $D^\beta_x D^\kappa_t f\in C^{\alpha,\alpha/2}(Q)$. Then a solution $u$ to \emph{($\star\star$)} satisfies $D^\beta_x D^\kappa_t u\in C^{\alpha,\alpha/2}(Q)$ for any $\beta, \kappa$ with $|\beta|+2\kappa\leq p+2,\, \kappa\leq q+1$.  In particular, $a^{ij},b^i,d\in C^{\infty}(B(0,r))$ and $f\in C^{\infty}(Q)$ imply that $u\in C^{\infty}(Q)$.\\

\emph{\textbf { b) Schauder estimates}} \\
\emph{(3)} Let $f\in C^\alpha(B(0,r))$. If $u\in C^2(B(0,r))$ satisfies \emph{($\star$)}
 then $u\in C^{2+\alpha}(B(0,r))$ and
\begin{align*}
|u|_{C^{1+\alpha}(B(0,r/2))} & \leq  C\big(|f|_{L^\infty(B(0,r))}+|u|_{L^\infty(B(0,r))}\big),\\
|u|_{C^{2+\alpha}(B(0,r/2))}& \leq  C\big(|f|_{C^\alpha(B(0,r))}+|u|_{L^\infty(B(0,r))}\big)
\end{align*} hold. Here, $C$ is a constant only determined by $n,\alpha,\Lambda/\lambda,|a^{ij}|_{C^\alpha(B(0,r))},$ $|b^{i}|_{C^\alpha(B(0,r))},|d|_{C^\alpha(B(0,r))}$.\\

\emph{(4)} Let $0\leq t\leq T$ and $f(\cdot,t)\in C^\alpha(B(0,r))$. If $u(\cdot,t)\in C^2(B(0,r))$ satisfies \emph{($\star\star$)} 
then $u(\cdot,t)\in C^{2+\alpha}(B(0,r))$ and
\begin{align*}
|u(\cdot,t)|_{C^{\alpha}(B(0,r/2))}
&\leq C\Big(\sup_{t\in [0,T]} |f(\cdot,t)|_{L^\infty(B(0,r))}+ \sup_{t\in [0,T]} |u(\cdot,t)|_{L^\infty(B(0,r))} \Big) 
\end{align*}\vspace{-20pt}
\begin{align*}
|u(\cdot,t)|_{C^{2+\alpha}(B(0,r/2))}+\Big|\dnd[u]{t}(\cdot,t)\Big|_{C^\alpha(B(0,r))}\\
&\mspace{-175mu}\leq C\Big(\sup_{t\in [0,T]} |f(\cdot,t)|_{C^\alpha(B(0,r))}+ \sup_{t\in [0,T]} |u(\cdot,t)|_{L^\infty(B(0,r))} \Big)       \end{align*} hold. Here, $C$ is a constant only determined by $n,\alpha,\Lambda/\lambda,|a^{ij}|_{C^\alpha(B(0,r))},$ $|b^{i}|_{C^\alpha(B(0,r))},|d|_{C^\alpha(B(0,r))}$.
\end{theo}

Concerning the above mentioned results see \cite{fried}, \cite{gilbarg}, \cite{lieber} and  \cite{wu}.

\begin{rem}\label{schaurem}The Schauder estimates are used in Section \ref{ltime}; the local estimates presented here carry over, e.g. by using a partition of unity, to the entire manifold $\Sigma$ in Proposition  \ref{fundesti}. 
\end{rem} \end{appendix}

\end{document}